\documentclass{article}%
\usepackage{graphicx}
\usepackage{amsmath}
\usepackage{amsfonts}
\usepackage{amssymb}%
\providecommand{\U}[1]{\protect\rule{.1in}{.1in}}

\newtheorem{theorem}{Theorem}
{}

\newtheorem{corollary}{Corollary}

\newtheorem{definition}{Definition}

\newtheorem{lemma}{Lemma}
{}
\newtheorem{notation}{Notation}

\newtheorem{proposition}{Proposition}
\newtheorem{remark}{Remark}

\newtheorem{summary}{Summary}
\newenvironment{proof}[1][Proof]{\textbf{#1.} }{\ \rule{0.5em}{0.5em}}

\begin{document}

\title{Spectral Expansion for the Non-self-adjoint Differential Operators with the
Periodic Matrix Coefficients}
\author{O. A. Veliev\\{\small Dogus University, \ Istanbul, Turkey. e-mail: oveliev@dogus.edu.tr}}
\date{}
\maketitle

\begin{abstract}
In this paper we construct the spectral expansion for the non-self-adjoint
differential operators generated in the space of vektor functions by the
ordinary differential expression of arbitrary order with the periodic matrix
coefficients by using the essential spectral singularities, singular
quasimomenta and the series with parenthesis.



\end{abstract}

\section{Introduction and Preliminary Results}

In this paper we construct a spectral expansion for the closed differential
operator $L$ generated in the space $L_{2}^{m}(-\infty,\infty)$ by the
differential expression
\begin{equation}
l(y)=y^{(n)}(x)+P_{2}\left(  x\right)  y^{(n-2)}(x)+P_{3}\left(  x\right)
y^{(n-3)}(x)+...+P_{n}(x)y, \tag{1}%
\end{equation}
where $n\geq2,$ $P_{\mathbb{\nu}}$ for $\mathbb{\nu}=2,3,...n$ are the
$m\times m$ matrix with the complex-valued entries $\left(  p_{\mathbb{\nu
},i,j}\right)  $ satisfying the conditions $p_{\mathbb{\nu},i,j}^{(n-v)}\in
L_{2}[0,1]$ and $p_{\mathbb{\nu},i,j}\left(  x+1\right)  =p_{\mathbb{\nu}%
,i,j}\left(  x\right)  $ for $i=1,2,...,m$ and $j=1,2,...,m.$ Moreover, we
assume that the eigenvalues $\mu_{1},\mu_{2},...,\mu_{m}$ of the matrix
\begin{equation}
C=\int_{0}^{1}P_{2}\left(  x\right)  dx \tag{2}%
\end{equation}
are simple. \ Here $L_{2}^{m}(I)$ denotes the space of the vector-valued
functions $f=\left(  f_{1},f_{2},...,f_{m}\right)  $ with the norm $\left\Vert
\cdot\right\Vert _{I}$ and inner product $(\cdot,\cdot)_{I}$ defined by%
\[
\left\Vert f\right\Vert _{I}^{2}=%
{\textstyle\int\limits_{I}}
\left\vert f\left(  x\right)  \right\vert ^{2}dx,\text{ }(f,g)_{I}=%
{\textstyle\int\limits_{I}}
\left\langle f\left(  x\right)  ,g\left(  x\right)  \right\rangle dx,
\]
where $\left\vert \cdot\right\vert $ and $\left\langle \cdot,\cdot
\right\rangle $ are the norm and inner product in $\mathbb{C}^{m},$ $f_{k}\in
L_{2}(I)$ for $k=1,2,...,m$ and the set $I$ is a subset of $\left(
-\infty,\infty\right)  .$ If $I=[0,1],$ then instead of $\left\Vert
\cdot\right\Vert _{I}$ and $(\cdot,\cdot)_{I}$ we write $\left\Vert
\cdot\right\Vert $ and $(\cdot,\cdot).$

Let us introduce some preliminary results and describe the scheme of the
paper. It is well-known that (see [4-6, 8]) the spectrum $\sigma(L)$ of $L$ is
the union of the spectra $\sigma(L_{t})$ of $L_{t}$ for $t\in\lbrack0,2\pi)$,
where $L_{t}$ is the differential operator generated in $L_{2}^{m}\left[
0,1\right]  $ by (1) and the boundary conditions
\begin{equation}
y\left(  1\right)  =e^{it}y\left(  0\right)  ,\text{ }y^{\prime}\left(
1\right)  =e^{it}y^{\prime}\left(  0\right)  ,\text{ }y^{\prime\prime}\left(
1\right)  =e^{it}y^{\prime\prime}\left(  0\right)  ,...,y^{(n-1)}\left(
1\right)  =e^{it}y^{(n-1)}\left(  0\right)  . \tag{3}%
\end{equation}
In [9] to obtain the asymptotic formulas for the operator $L_{t},$ the
operator $L_{t}(C)$ is taken for an unperturbed operator and $L_{t}-L_{t}(C)$
for a perturbation, where $L_{t}$ is denoted by $L_{t}(C)$ when $P_{2}(x)=C,$
$P_{3}(x)=0,...,P_{n}(x)=0.$ One can easily verify that the eigenvalues and
normalized eigenfunctions of $L_{t}(C)$ are%
\begin{equation}
\mu_{k,j}(t)=\left(  2\pi ki+ti\right)  ^{n}+\mu_{j}\left(  2\pi ki+ti\right)
^{n-2},\text{ \ }\Phi_{k,j,t}(x)=e(t)v_{j}e^{i\left(  2\pi k+t\right)  x}
\tag{4}%
\end{equation}
for $k\in\mathbb{Z}$ and $j=1,2,...,m,$ where $v_{1},v_{2},...,v_{m}$ are the
normalized eigenvectors of the matrix $C$ corresponding to the eigenvalues
$\mu_{1},\mu_{2},...,\mu_{m}$ respectively and
\[
\text{ }(e(t))=\left(
{\textstyle\int\nolimits_{0}^{1}}
\mid e^{itx}\mid^{2}dx\right)  ^{-\frac{1}{2}}.
\]
In this paper we essentially use the asymptotic formulas for the eigenvalues
and eigenfunctions of $L_{t}$ obtained in [9] for $t\in Q_{h},$ where
\begin{align*}
Q_{h}  &  =Q\backslash\left(  U_{h}(0)\cup U_{h}(\pi)\right)  ,\\
\text{ }Q  &  =\{t\in\mathbb{C}:-h<\operatorname{Re}t\leq2\pi-h,\text{
}|\operatorname{Im}t|\leq h\mathbb{\}},\\
U_{h}(z)  &  =\{t\in\mathbb{C}:\text{ }|t-z|<h\mathbb{\}}%
\end{align*}
and $0<h<\frac{1}{15\pi}.$ More precisely, we use the following results of [9]
formulated here as the following summary.

\begin{summary}
There exists a positive constant $N(h)$ such that if $t\in Q_{h},$ then

$(a)$ The large eigenvalues of $L_{t}$ consist of $m$ sequences
\[
\left\{  \lambda_{k,1}(t):\mid k\mid\geq N(h)\right\}  ,\left\{  \lambda
_{k,2}(t):\mid k\mid\geq N(h)\right\}  ,...,\left\{  \lambda_{k,m}(t):\mid
k\mid\geq N(h)\right\}
\]
satisfying the following formulas
\begin{equation}
\lambda_{k,j}(t)=\left(  2\pi ki+ti\right)  ^{n}+\mu_{j}\left(  2\pi
ki+ti\right)  ^{n-2}+O(k^{n-3}\ln|k|) \tag{5}%
\end{equation}
for $j=1,2,...,m$ as $|k|\rightarrow\infty$. This formula is uniform with
respect to $t$ in $Q_{h},$ that is, there exists a constant $c,$ independent
of $t,$ such that the term $O(k^{n-3}\ln|k|)$ in (5) satisfies the inequality%
\[
\mid O(k^{-1}\ln|k|)\mid<c\mid k^{-1}\ln|k|\mid
\]
for all $t\in Q_{h}$ and $\mid k\mid\geq N(h).$ Apart from the eigenvalues
$\lambda_{k,j}(t),$ defined in (5), there exist $K(h)$ eigenvalues of the
operator $L_{t}$ denoted by $\lambda_{k}(t)$ for $k=1,2,...,K(h),$ where
$K(h)=(2N(h)-1)m.$ Moreover, there exists a finite set $A_{h}$ such that all
eigenvalues of the operators $L_{t}$ for $t\in Q_{h}\backslash A_{h}$ are simple.

$(b)$ If $\mid k\mid\geq N(h),$ then $\lambda_{k,j}(t)$ is a simple eigenvalue
of $L_{t}$ and the corresponding normalized eigenfunction $\Psi_{k,j,t}(x)$
satisfies
\begin{equation}
\Psi_{k,j,t}(x)=e(t)v_{j}e^{i\left(  2\pi k+t\right)  x}+O(k^{-1}\ln|k|).
\tag{6}%
\end{equation}
This formula is uniform with respect to $t$ and $x$ in $Q_{h}$ and in $[0,1]$
respectively. In other words, there exists a constant $c,$ independent of $t$
and $x$ such that the term $O(k^{-1}\ln|k|)$ in (6) satisfies%
\[
\mid O(k^{-1}\ln|k|)\mid<c\mid k^{-1}\ln|k|\mid
\]
for all $t\in Q_{h},$ $x\in\lbrack0,1]$ and $\mid k\mid\geq N(h).$

$(c)$ The root functions of $L_{t}$ form a Riesz basis in $L_{2}^{m}(0,1)$.\ 

$(d)$ Let $L_{t}^{\ast}$ be the adjoint operator of $L_{t}$ and $X_{k,j,t}$ be
the eigenfunction of $L_{t}^{\ast}$ corresponding to the eigenvalue
$\overline{\lambda_{k,j}(t)}$ and satisfying $(X_{k,j,t},\Psi_{k,j,t})=1$,
where $\mid k\mid\geq N(h)$ and $t\in Q_{h}.$ Then $X_{k,j,t}(x)$ satisfies
the following formula
\begin{equation}
X_{k,j,t}(x)=u_{j}(e(t))^{-1}e^{i(2k\pi+\bar{t})x}+O(k^{-1}\ln|k|), \tag{7}%
\end{equation}
where $u_{j}$ is the eigenvector of \ $C^{\ast}$ corresponding to
$\overline{\mu_{j}}$ and satisfying $\left\langle u_{j},v_{j}\right\rangle
=1.$ This formula is uniform with respect to $t$ and $x$ in $Q_{h}$ and in
$[0,1]$ respectively. For each $k$ and $x$ the functions $\lambda_{k}(t),$
$\lambda_{k,j}(t)$ $X_{k,t}(x),\Psi_{k,t}(x),X_{k,j,t}(x)$ and $\Psi
_{k,j,t}(x)$ continuously depend on $t$ at $Q_{h}$ exsept at most a finite
number of points, where $\Psi_{k,t}$ and $X_{k,t}$ are respectively the
eigenfunctions of $L_{t}$ and $L_{t}^{\ast}$ corresponding to $\lambda_{k}(t)$
and $\overline{\lambda_{k}(t)}$ and satisfying $(X_{k,t},\Psi_{k,t})=1.$
\end{summary}

Then we use the vectorial\ form of the Gelfand transform [1] formulated as
follows. For every $f\in L_{2}^{m}(-\infty,\infty)$ there exists $f_{t}(x)$
such that
\begin{equation}
f(x)=\frac{1}{2\pi}%
{\displaystyle\int\limits_{[0,2\pi)}}
f_{t}(x)dt,\text{ }f_{t}(x)=\sum\limits_{k=-\infty}^{\infty}f(x+k)e^{-ikt}
\tag{8}%
\end{equation}
and%
\begin{equation}
\text{ }%
{\displaystyle\int\limits_{(-\infty,\infty)}}
\left\vert f(x)\right\vert ^{2}dx=\frac{1}{2\pi}%
{\displaystyle\int\limits_{[0,2\pi)}}
{\displaystyle\int\limits_{[0,1]}}
\left\vert f_{t}(x)\right\vert ^{2}dxdt,\text{ }f_{t}(x+1)=e^{it}f_{t}(x).
\tag{9}%
\end{equation}

We use Summary 1 and formula (8) as follows. By Summary 1$(c)$ the root
functions of $L_{t}$ for $t\neq0,\pi$ form a Riesz basis in $L_{2}^{m}(0,1)$.
Moreover, it follows from Summary 1$(a)$ that there exists a countable set $A$
such that the eigenvalues of the operators $L_{t}$ for $t\in Q\backslash A$
are simple. Without loss of generality, we may assume that $\left\{
0,\pi\right\}  \subset A.$ Thus if $t\in Q\backslash A,$ then the root
functions of $L_{t}$ consist of the eigenfunctions, denoted here and
occasionally, for the simplicity of the notations, by $\Psi_{k,t}$ for
$k=1,2,....$ Therefore we have the following decomposition
\begin{equation}
f_{t}=\sum_{k=1}^{\infty}(f_{t},X_{k,t})\Psi_{k,t}, \tag{10}%
\end{equation}
where $X_{k,t}$ is the eigenfunction of $L_{t}^{\ast}$ satisfying the equality
$\left(  \Psi_{k,t},X_{k,t}\right)  =1$ and $t\in Q\backslash A.$ Using (10)
in (8) we obtain
\begin{equation}
f(x)=\frac{1}{2\pi}%
{\displaystyle\int\limits_{[0,2\pi)}}
\sum_{k=1}^{\infty}(f_{t},X_{k,t})\Psi_{k,t}(x)dt. \tag{11}%
\end{equation}
Thus, in order to get the spectral expansion we need to consider the term by
term integration of the series in (11).

In Section 2 we prove that the integrand of (11) can be integrated term by
term if we consider the following curve of integration and the following
subset of $L_{2}^{m}(-\infty,\infty).$ It follows from Summary 1$(a)$ that the
accumulation points of the set $A$ are $0$ and $\pi.$ Moreover, the set $A\cap
Q_{h\text{ }}$ is finite. Denote the points of $A\cap B(h)$ by $t_{1}%
,t_{2},...,t_{v(h)}$, where $B(h)=[0,2\pi)\cap Q_{h\text{ }}=[h,\pi
-h]\cup\lbrack\pi+h,2\pi-h]$ and
\[
h<t_{1}<t_{2}<...<t_{p}<\pi-h<\pi+h<t_{p+1}<t_{p+2}<\cdot\cdot\cdot
<t_{v(h)}<2\pi-h.
\]
Instead of the interval $[0,2\pi)$ we choose the curve of integration so that
it passes over the points $0,t_{1},t_{2},...,t_{v(h)}$ and $\pi$. Namely,
instead of $[0,2\pi)$ we consider the curve $l_{+}(h,\varepsilon(h))$ joining
the points $-h$ and $2\pi-h$ and consisting of the intervals
\begin{align}
&  [h,t_{1}-\varepsilon(h)],[t_{1}+\varepsilon(h),t_{2}-\varepsilon
(h)],...,[t_{p-1}+\varepsilon(h),t_{p}-\varepsilon(h)],[t_{p}+\varepsilon
(h),\pi-h],\tag{12}\\
&  [\pi+h,t_{p+1}-\varepsilon(h)],[t_{p+1}+\varepsilon(h),t_{p+2}%
-\varepsilon(h)],...,[t_{v(h)}+\varepsilon(h),2\pi-h]\nonumber
\end{align}
and semicircles
\begin{align}
\gamma_{+}(0,h)  &  =\{\left\vert t\right\vert =h,\operatorname{Im}%
t>0\},\text{ }\gamma_{+}(\pi,h)=\{\left\vert t-\pi\right\vert
=h,\operatorname{Im}t>0\},\tag{13}\\
\gamma_{+}(t_{j},\varepsilon(h))  &  =\{\left\vert t-t_{j}\right\vert
=\varepsilon(h),\operatorname{Im}t>0\}\nonumber
\end{align}
for $j=1,2,...,v(h),$ where $\varepsilon(h)<\frac{h}{v(h)}.$ Note that instead
of $[0,2\pi)$ one can consider $(-h,2\pi-h],$ since $L_{t+2\pi}=L_{t}$.
Moreover, $h$ and $\varepsilon(h)$ can be chosen so that the curves
$l_{+}(h,\varepsilon)$ and $l_{-}(h,\varepsilon):=\{\lambda:\overline{\lambda
}\in l_{+}(h,\varepsilon)\}$ do not contain the points of the set $A$ and the
disk $\{\left\vert t-t_{j}\right\vert <\varepsilon(h)\}$ contains only the
point $t_{j}$ from $A,$ where $j=1,2,...,v(h).$ Finally, note that
$\varepsilon(h)$ can be chosen so that the semicircles (13) have no common points.

In Section 2 we prove that the integrand of (11) can be integrated term by
term if we replace the integration set $[0,2\pi)$ by the curve $l_{\pm
}(h,\varepsilon(h))$. However, equality obtained from (11) by changing
$[0,2\pi)$ to $l_{\pm}(h,\varepsilon)$ holds if and only if
\[%
{\displaystyle\int\limits_{\lbrack0,2\pi)}}
f_{t}(x)dt=%
{\displaystyle\int\limits_{l_{\pm}(h,\varepsilon)}}
f_{t}(x)dt
\]
(see (8) and (10)). Therefore, we consider the functions $f\in L_{2}%
^{m}(-\infty,\infty)$ for which the last equality holds in sense of (14) (see
below). For this we consider the following subset $W$ of $L_{2}^{m}%
(-\infty,\infty)$.

\begin{definition}
Let $W$ be the set of $f\in L_{2}^{m}(-\infty,\infty)$ such that $f=$
$f^{+}+f^{-}$, where $f^{+}\in W^{+},$ $f^{-}\in W^{-}$ and the sets $W^{+}$
and $W^{-}$ are defined as follows. Denote by $D_{+}(h,\varepsilon)$ the open
set enclosed by $l_{+}(h,\varepsilon)$ and $[-h,2\pi-h],$ where $l_{+}%
(h,\varepsilon)$ is the union of the intervals (12) and semicircles (13),
however $\varepsilon$ does not depend on $h.$ Let $W^{+}$ be the set of
$f^{+}\in L_{2}(-\infty,\infty)$ for which there exists $\varepsilon>0$ such
that the Gelfand transform defined by%
\[
f_{t}^{+}(x)=\sum_{k=-\infty}^{\infty}f^{+}(x+k)e^{-ikt}%
\]
is continuous function of two variables $\left(  x,t\right)  $ in
\[
\left(  \overline{D_{+}(h,\varepsilon)}\cup\lbrack-h,2\pi-h]\right)
\times\lbrack0,1]
\]
and analytic in $D_{+}(h,\varepsilon)$ with respect to $t$ for almost all
$x\in\lbrack0,1].$ Replacing $D_{+}(h,\varepsilon)$ by
\[
D_{-}(h,\varepsilon)=\{\lambda:\overline{\lambda}\in D_{+}(h,\varepsilon)\}
\]
and repeating the definition of \ $W^{+}$ we get the definition of $W^{-}.$
\end{definition}

Note that the set $W$ is a sufficiently large subset of $L_{2}^{m}%
(-\infty,\infty).$ For example, $W$ contains the set of $f\in L_{2}%
^{m}(-\infty,\infty)$ such that $f$ is continuous on $[k,k+1)$ for each
$k\in\mathbb{Z}$ and the series
\[
\sum_{k=-\infty}^{\infty}\mid f(x+k)\mid
\]
converges uniformly in $[0,1]$. This can be easily verified as follows. Define
$f^{+}$ and $f^{-}$ by%
\[
f^{+}(x)=f(x),\text{ if }x\in(-\infty,0),\text{ }f^{+}(x)=0\text{ if }%
x\in\lbrack0,\infty)
\]
and%
\[
f^{-}(x)=f(x),\text{ if }x\in\lbrack0,\infty),\text{ }f^{-}(x)=0\text{ if
}x\in(-\infty,0).
\]
Then $f=$ $f^{+}+f^{-}$ and $f^{\pm}\in W^{\pm}$, since the last series
converges uniformly in $[0,1]$.

It follows from Definition 1 that if $f\in W,$ then
\begin{equation}
\int\limits_{l_{+}(h,\varepsilon)}f_{t}^{+}(x)dt=\int\limits_{(-h,2\pi
-h]}f_{t}^{+}(x)dt,\text{ }\int\limits_{l_{-}(h,\varepsilon)}f_{t}%
^{-}(x)dt=\int\limits_{(-h,2\pi-h]}f_{t}^{-}(x)dt. \tag{14}%
\end{equation}
By the notation of Summary 1 for $t\in Q_{h}\backslash A$ the root functions
of $L_{t}$ consist of the eigenfunctions $\Psi_{k,t}$ for $k=1,2,...,K(h)$ and
$\Psi_{k,j,t}$ for $j=1,2,...,m$ and $|k|\geq N(h).$ Therefore we have
\begin{equation}
f_{t}^{+}(x)=\sum_{k=1}^{K(h)}(f_{t}^{+},X_{k,t})\Psi_{k,t}(x)+\sum
_{\left\vert k\right\vert \geq N(h)}^{{}}\sum_{j=1}^{m}(f_{t}^{+}%
,X_{k,j,t})\Psi_{k,j,t}(x) \tag{10a}%
\end{equation}
for all $t\in Q_{h}\backslash A.$ Note that equality (10a) is (10) rewritten
in the notation of Summary 1. This equality with (14), (8) and the following
obvious equality $f_{t+2\pi}(x)=f_{t}(x)$ implies that
\begin{equation}
f^{+}(x)=\frac{1}{2\pi}\int\limits_{l_{+}(\varepsilon,h)}%
{\displaystyle\sum\limits_{k=1}^{K(h)}}
(f_{t}^{+},X_{k,t})\Psi_{k,t}(x)+\sum_{\left\vert k\right\vert \geq N(h)}^{{}%
}\sum_{j=1}^{m}(f_{t}^{+},X_{k,t})\Psi_{k,t}(x)dt. \tag{15}%
\end{equation}
In the same way we obtain
\begin{equation}
f^{-}(x)=\frac{1}{2\pi}\int\limits_{l_{-}(\varepsilon,h)}%
{\displaystyle\sum\limits_{k=1}^{K(h)}}
(f_{t}^{-},X_{k,t})\Psi_{k,t}(x)+\sum_{\left\vert k\right\vert \geq N(h)}^{{}%
}\sum_{j=1}^{m}(f_{t}^{-},X_{k,t})\Psi_{k,t}(x)dt. \tag{16}%
\end{equation}

In Section 2 using (15) and (16) we prove that
\begin{equation}
f^{+}=\frac{1}{2\pi}\left(  \sum_{k=1}^{K(h)}%
{\displaystyle\int\limits_{l_{+}(h,\varepsilon)}}
(f_{t}^{+},X_{k,t}),\Psi_{k,t}dt+\sum_{\left\vert k\right\vert \geq N(h)}^{{}%
}\sum_{j=1}^{m}%
{\displaystyle\int\limits_{l_{+}(h,\varepsilon)}}
(f_{t}^{+},X_{k,j,t})\Psi_{k,j,t}dt\right)  \tag{17}%
\end{equation}
and%
\begin{equation}
f^{-}=\frac{1}{2\pi}\left(  \sum_{k=1}^{K(h)}%
{\displaystyle\int\limits_{l_{-}(h,\varepsilon)}}
(f_{t}^{-},X_{k,t}),\Psi_{k,t}dt+\sum_{\left\vert k\right\vert \geq N(h)}^{{}%
}\sum_{j=1}^{m}%
{\displaystyle\int\limits_{l_{-}(h,\varepsilon)}}
(f_{t}^{-},X_{k,j,t})\Psi_{k,j,t}dt\right)  . \tag{18}%
\end{equation}
for all $f\in W.$ In other words, we solve the term by term integration
problem for the curve $l_{\pm}(h,\varepsilon).$ Note that in papers [9] such
term by term integration problem for the operator $L$ was solved only for the
compactly supported continuous function. Moreover, in Section 2 using (17) and
(18) we prove the following.

\begin{theorem}
If $f\in W,$ then for each $p>0$ and $\delta>0$ there exists $h>0$ such that
\[
\left\Vert f-\frac{1}{2\pi}\left(  \sum_{k=1}^{K(h)}%
{\displaystyle\int\limits_{r(h,\varepsilon)}}
(f_{t},X_{k,t})\Psi_{k,t}dt+\sum_{\left\vert k\right\vert \geq N(h)}^{{}}%
\sum_{j=1}^{m}%
{\displaystyle\int\limits_{r(h,\varepsilon)}}
(f_{t},X_{k,j,t})\Psi_{k,j,t}dt\right)  \right\Vert _{(-p,p)}<\delta,
\]
where $r(h,\varepsilon(h))$ is the real part of $l_{\pm}(h,\varepsilon(h)),$
that is, is the union of the intervals in (12).
\end{theorem}

We say that Theorem 1 is the approximation for the spectral expansion, since
$r(h,\varepsilon(h))\rightarrow\lbrack0,2\pi)$ as $h\rightarrow0$. Similarly
formulas (17) and (18) can be considered as approximations for the spectral expansion.

In Section 3 using (17), (18) and the notion essential spectral singularity
(ESS) defined in [10, 12] for the scalar case $m=1$ we construct the spectral
expansion for the operator $L.$ To get the spectral expansion in the term of
$t$ from (17) and (18) we need replace the integrals over $l_{\pm}%
(\varepsilon,h)$ by the integrals over $(-h,2\pi-h]$ and consider the
integrability of $(f_{t},X_{k,t}),\Psi_{k,t}$ over $(-h,2\pi-h]$. It is
possible that the function $(f_{t},X_{k,t}),\Psi_{k,t}$ for some value of
$k$\ is not integrable on $(-h,2\pi-h],$ since
\[
(f_{t},X_{k,t})\Psi_{k,t}=\frac{1}{\alpha_{k}(t)}(f_{t},\Psi_{k,t}^{\ast}%
)\Psi_{k,t}%
\]
and $\alpha_{k}(t)=0$ if there exists an associated function corresponding to
$\Psi_{k,t}$ and hence $\frac{1}{\alpha_{k}}$ may become nonintegrable, where
$\alpha_{k}(t)=(\Psi_{k,t}^{\ast},\Psi_{k,t})$ and $\Psi_{k,t}^{\ast}$ is a
normalized eigenfunction of $L_{t}^{\ast}$ corresponding to the eigenvalue
$\overline{\lambda_{k}(t)}.$ Therefore we introduce the following notions.

\begin{definition}
A number $\lambda_{0}\in\sigma(L)$ is said to be an essential spectral
singularity (ESS) of $L$ if there exist $t_{0}\in(-h,2\pi-h]$ and
$k\in\mathbb{Z}$ such that $\lambda_{0}=\lambda_{k}(t_{0})$ and $\frac
{1}{\alpha_{k}}$ is not integrable over $(t_{0}-\delta,t_{0}+\delta)$ for all
$\delta>0.$ Then $t_{0}$ is called a singular quasimomentum (SQ).
\end{definition}

If $\lambda_{k}(t_{0})$ is a simple eigenvalue, then $\alpha_{k}(t)\neq0$ and
$\frac{1}{\alpha_{k}}$ is integrable over $(t_{0}-\delta,t_{0}+\delta)$ for
some $\delta>0.$ Therefore ESS is a multiple eigenvalue and the set of
singular quasimomenta is a subset of $A.$ Thus if $\lambda_{k}(t_{j})$ is an
ESS for some $k,$ then the function $(f_{t},X_{k,t})\Psi_{k,t}(x),$ in
general, is not integrable on $(t_{j}-\delta,t_{j}+\delta)$ for all $\delta>0$
and for some $f\in L_{2}^{m}(-\infty,\infty)$ and $x\in\lbrack0,1].$
Fortunately, the sum
\[
\sum_{s\in\mathbb{T}(k,j)}(f_{t},X_{s,t})\Psi_{s,t},
\]
where $\mathbb{T}(k,j):=\left\{  s:\lambda_{s}(t_{j})=\lambda_{k}%
(t_{j})\right\}  ,$ is integrable over $(t_{j}-\delta,t_{j}+\delta)$ for some
$\delta>0.$ Thus we need to compound together some nonintegrable summand of
the series (10). This situation requires to use the brackets in the spectral
expansion. In Section 3 we construct spectral expansion with brackets for
every $f\in W.$ Note that if $m$ is an odd number then $A$ is an finite set
and formulas (5)-(7) hold uniformly with respect to $t$ and $x$ in $Q$ and in
$[0,1]$ respectively. Therefore it is easy to get the term by term integration
in (11). This simpler case is investigated in [11], where we proved that if
$m$ is an odd number then the operator $L$ is an asymptotically spectral
operator and has an elegant spectral expansion. \ Here we consider the
complicated case when $m$ is an even number, since in this case there may
exist infinitely many singular quasimamenta and the set of indices $(k,j)$ for
which $(f_{t}^{\pm},X_{k,j,t})\Psi_{k,j,t}(x)$ are not integrable may contain
infinitely many elements. Moreover, this case is interesting, since it
contains the Schr\"{o}dinger operator with a matrix potential. The results of
this paper are new for the scalar case $m=1$ too, since in [12] we considered
essentially small subset of $L_{2}(-\infty,\infty)$ than $W$ $.$ Besides, in
this paper we minimize the number of terms in the bracket used in the spectral
expansion. This is explained in detail in the conclusion Section 4. Finally
note that this paper together with the papers [10-12] gives the complete
solution of the spectral expansion problem for the ordinary differential
operators with the periodic coefficients. Therefore, this paper can be
considered as continuation and completion of the papers [10-12].

\section{On the Approximations for the Spectral Expansion}

In this section we prove that for $f$ $\in W$ equalities (17) and (18) hold.
Moreover, we give the proof of Theorem 1. To prove (17) and (18) it is enough
to prove the term by term integrations in (15) and (16) respectively. We prove
(17). The proof of (18) is the same. For this we consider the remainder
\begin{equation}
\sum_{\left\vert k\right\vert \geq s}^{{}}\sum_{j=1}^{m}(f_{t}^{+}%
,X_{k,j,t})\Psi_{k,j,t}(x) \tag{19}%
\end{equation}
of the series
\begin{equation}
\sum_{\left\vert k\right\vert \geq N(h)}^{{}}\sum_{j=1}^{m}(f_{t}%
^{+},X_{k,j,t})\Psi_{k,j,t}(x), \tag{20}%
\end{equation}
where $s\geq N(h)$. In the forthcoming inequalities we denote by $c_{1}%
,c_{2},...$ the positive constants that do not depend on $t$. They will be
used in the sense that there exists $c_{i}$ such that the inequality holds.

\begin{lemma}
If $f\in W,$ then
\begin{equation}
\left\Vert \sum_{\left\vert k\right\vert \geq s}^{{}}\sum_{j=1}^{m}(f_{t}%
^{+},X_{k,j,t})\Psi_{k,j,t}\right\Vert ^{2}\leq c_{1}\left(  \sum_{\left\vert
k\right\vert \geq s}^{{}}\sum_{j=1}^{m}\mid(e^{-itx}f_{t}^{+},e_{j}e^{i2\pi
kx})\mid^{2}+\frac{\left\Vert f_{t}^{+}\right\Vert ^{2}}{\sqrt{s}}\right)
\tag{21}%
\end{equation}
for all $s\geq N(h),$ $t\in l_{+}(\varepsilon,h),$ where $\left\{
e_{j}:j=1,2,...,m\right\}  $\ is a standard basis of $\mathbb{C}^{m}$.
\end{lemma}

\begin{proof}
We prove (21) by showing that the inequalities
\begin{equation}
\sum_{\left\vert k\right\vert \geq s}^{{}}\sum_{j=1}^{m}\mid(f_{t}%
^{+},X_{k,j,t})\mid^{2}\leq c_{2}\left(  \sum_{\left\vert k\right\vert \geq
s}^{{}}\sum_{j=1}^{m}\mid(e^{-itx}f_{t}^{+},e_{j}e^{i2\pi kx})\mid^{2}%
+\frac{1}{\sqrt{s}}\parallel f_{t}^{+}\parallel^{2}\right)  \tag{22}%
\end{equation}
and%
\begin{equation}
\left\Vert \sum_{\left\vert k\right\vert \geq s}^{{}}\sum_{j=1}^{m}(f_{t}%
^{+},X_{k,j,t})\Psi_{k,j,t}(x)\right\Vert ^{2}\leq c_{3}\sum_{\left\vert
k\right\vert \geq s}^{{}}\sum_{j=1}^{m}\mid(f_{t}^{+},X_{k,t})\mid^{2}
\tag{23}%
\end{equation}
hold. First let us prove (22). Using (7) we obtain
\begin{equation}
\mid(f_{t}^{+},X_{k,j,t})\mid^{2}\leq c_{4}\left(  \mid(e^{-itx}f_{t}%
^{+},u_{j}e^{i2\pi kx})\mid^{2}+\parallel f_{t}^{+}\parallel^{2}\left\vert
k^{-1}\ln|k|\right\vert ^{2}\right)  \tag{24}%
\end{equation}
for $\left\vert k\right\vert \geq N(h)$. Since $u_{j}\in\mathbb{C}^{m}$ and
$e_{1},e_{1},...,e_{m}$\ is a standard basis in $\mathbb{C}^{m}$ we have
\begin{equation}
\sum_{\left\vert k\right\vert \geq s}^{{}}\sum_{j=1}^{m}\mid(e^{-itx}f_{t}%
^{+},u_{j}e^{i2\pi kx})\mid^{2}\leq c_{5}\sum_{\left\vert k\right\vert \geq
s}^{{}}\sum_{j=1}^{m}\mid(e^{-itx}f_{t}^{+},e_{j}e^{i2\pi kx})\mid^{2}.
\tag{25}%
\end{equation}
Thus (22) follows from (24) and \ (25).

Now we prove (23). For this we use the relations%
\begin{equation}
\Psi_{k,j,t}(x)=e(t)v_{j}e^{i(2\pi k+t)x}\text{ }+h_{k,j,t}(x),\text{
}\left\Vert h_{k,j,t}\right\Vert =O(k^{-1}\ln|k|) \tag{26}%
\end{equation}
(see (6)) and%
\begin{equation}
\left\Vert (f_{t}^{+},X_{k,j,t})h_{k,j,t}(x)\right\Vert \leq c_{6}\left(
\mid(f_{t}^{+},X_{k,j,t})\mid^{2}+\left\vert k^{-1}\ln|k|\right\vert
^{2}\right)  \tag{27}%
\end{equation}
for $\left\vert k\right\vert \geq N(h)$. Using (24), (25) and the Bessel
inequality for the orthonormal basis
\begin{equation}
\left\{  e_{i}e^{i2\pi kx}:k\in\mathbb{Z};\text{ }i=1,2,...,m\right\}
\tag{28}%
\end{equation}
of $L_{2}^{m}(0,1)$ we obtain
\begin{equation}
\sum_{\left\vert k\right\vert \geq s}^{{}}\sum_{j=1}^{m}\mid(f_{t}%
^{+},X_{k,j,t})\mid^{2}\leq c_{7}\left(  \parallel e^{-itx}f_{t}^{+}%
\parallel^{2}+\parallel f_{t}^{+}\parallel^{2}\right)  . \tag{29}%
\end{equation}
This and (27) imply that the series
\[
\sum_{\left\vert k\right\vert \geq s}^{{}}\sum_{j=1}^{m}(f_{t}^{+}%
,X_{k,j,t})v_{j}e^{i(2\pi k+t)x}\text{ \ }%
\]
and
\[
\sum_{\left\vert k\right\vert \geq s}^{{}}\sum_{j=1}^{m}(f_{t}^{+}%
,X_{k,j,t})h_{k,j,t}(x)
\]
converge in the norm of $L_{2}^{m}(0,1).$ Therefore by (26) we have
\begin{equation}
\left\Vert \sum_{\left\vert k\right\vert \geq s}^{{}}\sum_{j=1}^{m}(f_{t}%
^{+},X_{k,j,t})\Psi_{k,j,t}(x)\right\Vert ^{2}\leq2S_{1}+2S_{2}^{2}, \tag{30}%
\end{equation}
where
\begin{equation}
S_{2}=\left\Vert \sum_{\left\vert k\right\vert \geq s}^{{}}\sum_{j=1}%
^{m}(f_{t}^{+},X_{k,j,t})h_{k,j,t}\right\Vert \tag{31}%
\end{equation}
and%
\begin{equation}
S_{1}=\left\Vert \sum_{\left\vert k\right\vert \geq s}^{{}}\sum_{j=1}%
^{m}(f_{t}^{+},X_{k,j,t})v_{j}e(t)e^{i(2\pi k+t)x}\text{ }\right\Vert ^{2}\leq
c_{8}\sum_{\left\vert k\right\vert \geq s}^{{}}\sum_{j=1}^{m}\mid(f_{t}%
^{+},X_{k,j,t})\mid^{2}. \tag{32}%
\end{equation}
\ Now let us estimate $S_{2}.$ It follows from the second equality of (26)
that
\[
S_{2}\leq c_{9}\sum_{\left\vert k\right\vert \geq s}^{{}}\sum_{j=1}^{m}%
\mid(f_{t}^{+},X_{k,j,t})\mid k^{-1}\ln|k|.
\]
Now using the Schwarz inequality for $l_{2}$ we obtain
\begin{equation}
S_{2}^{2}=\left(  \sum_{\left\vert k\right\vert \geq s}^{{}}\sum_{j=1}^{m}%
\mid(f_{t}^{+},X_{k,j,t})\mid^{2}\right)  O(s^{-\frac{1}{2}}). \tag{33}%
\end{equation}
Therefore (23) follows from (30)-(33). Thus (22) and (23) and hence (21) is proved.
\end{proof}

\begin{theorem}
If $f$ $\in W,$ then (17) and (18) are satisfied. The series in (17) and (18)
converge in the $L_{2}^{m}(a,b)$ norm for any $a,b\in\mathbb{R}.$
\end{theorem}

\begin{proof}
Using the definition of $l_{+}(\varepsilon,h),$ Definition 1 and Summary 1 one
can easily verify that the functions
\[
f_{t}^{+}(x),\text{ }(f_{t}^{+},X_{k,t})\Psi_{k,t}(x),\text{ }(f_{t}%
^{+},X_{k,j,t})\Psi_{k,j,t}(x)
\]
are integrable with respect to $t$ on the curve $l_{+}(\varepsilon,h).$
Therefore from (15) we obtain%
\[
f^{+}(x)=\frac{1}{2\pi}%
{\displaystyle\sum\limits_{k=1}^{K(h)}}
\int\limits_{l_{+}(\varepsilon,h)}(f_{t}^{+},X_{k,t})\Psi_{k,t}(x)dt+\sum
_{\left\vert k\right\vert <s}^{{}}\sum_{j=1}^{m}\int\limits_{l_{+}%
(\varepsilon,h)}(f_{t}^{+},X_{k,t})\Psi_{k,t}(x)dt+R_{s}^{+}(x),
\]
where
\[
R_{s}^{+}(x)=\int\limits_{l_{+}(\varepsilon,h)}\sum_{\left\vert k\right\vert
\geq s}^{{}}\sum_{j=1}^{m}(f_{t}^{+},X_{k,j,t})\Psi_{k,j,t}(x)dt
\]
and $s>N(h).$ Thus to prove (17) it is enough to show that%
\begin{equation}
\left\Vert R_{s}^{+}\right\Vert _{(-p,p)}\rightarrow0 \tag{34}%
\end{equation}
as $s\rightarrow\infty$ for each $p>0.$

Now we prove (34). Since (28) is an orthonormal basis in $L_{2}^{m}(0,1),$ we
have
\begin{equation}
\sum_{k\in\mathbb{Z}}^{{}}\sum_{j=1}^{m}\mid(e^{-itx}f_{t}^{+},e_{j}e^{i2\pi
kx})\mid^{2}=\parallel e^{-itx}f_{t}^{+}\parallel^{2}. \tag{35}%
\end{equation}
Moreover, using Definition 1 we obtain that $\mid(e^{-itx}f_{t}^{+}%
,e_{j}e^{i2\pi kx})\mid^{2}$ for $k\in\mathbb{Z}$ , $j=1,2,...,m$ and
$\parallel e^{-itx}f_{t}^{+}\parallel^{2}$ are the continuous function on the
compact $l_{+}(\varepsilon,h).$ Therefore the series in (35) converges
uniformly on $l_{+}(\varepsilon,h).$ This statement with (21) implies that
\begin{equation}
\left\Vert G_{s}^{+}(t,\cdot)\right\Vert ^{2}\rightarrow0 \tag{36}%
\end{equation}
uniformly on $l_{+}(\varepsilon,h)$ as $s\rightarrow\infty,$ where
\[
G_{s}^{+}(t,x)=\sum_{\left\vert k\right\vert \geq s}^{{}}\sum_{j=1}^{m}%
(f_{t}^{+},X_{k,j,t})\Psi_{k,j,t}(x).
\]
Thus using the obvious equality
\[
\Psi_{k,j,t}(x+1)=e^{it}\Psi_{k,j,t}(x)
\]
we obtain
\begin{equation}
\left\Vert G_{s}^{+}(t,\cdot)\right\Vert _{(-p,p)}^{2}\leq b_{s}(p) \tag{37}%
\end{equation}
for all $t\in l_{+}(\varepsilon,h),$ where $b_{s}(p)\rightarrow0$ as
$s\rightarrow\infty$ for all $p>0.$ It implies that
\[
\int\limits_{l_{+}(\varepsilon,h)}\int\limits_{(-p,p)}\left\vert G_{s}%
^{+}(t,x)\right\vert ^{2}dxdt
\]
exists.

Now we consider the integral over $l_{+}(h,\varepsilon)$ as sum of the
integrals over $r(h,\varepsilon)=l_{+}(h,\varepsilon)\cap\mathbb{R}$ and the
semicircles (13). Thus $R_{s}^{+}(x)$ is the sum of the integrals
\begin{equation}
\int\limits_{r(h,\varepsilon)}G_{s}^{+}(t,x)dt,\int\limits_{\gamma_{+}%
(0,h)}G_{s}^{+}(t,x)dt,\int\limits_{\gamma_{+}(\pi,h)}G_{s}^{+}(t,x)dt,\int
\limits_{\gamma_{+}(t_{j},h)}G_{s}^{+}(t,x)dt \tag{38}%
\end{equation}
for $j=1,2,...,v(h).$ First let us consider the integral over $r(h,\varepsilon
).$ Using the obvious inequality%
\[
\left\vert \int\limits_{r(h,\varepsilon)}f(t)dt\right\vert ^{2}\leq c_{10}%
\int\limits_{r(h,\varepsilon)}\left\vert f(t)\right\vert ^{2}dt,
\]
Fubuni theorem and (37) we obtain%
\[
\int\limits_{(-p,p)}\left\vert \int\limits_{r(h,\varepsilon)}G_{s}%
^{+}(t,x)dt\right\vert ^{2}dx\leq c_{10}\int\limits_{(-p,p)}\int
\limits_{r(h,\varepsilon)}\left\vert G_{s}^{+}(t,x)\right\vert ^{2}dtdx=
\]%
\begin{equation}
c_{10}\int\limits_{r(h,\varepsilon)}\int\limits_{(-p,p)}\left\vert G_{s}%
^{+}(t,x)\right\vert ^{2}dxdt=c_{10}\int\limits_{r(h,\varepsilon)}\left\Vert
G_{s}^{+}(t,\cdot)\right\Vert _{(-p,p)}^{2}dt\leq c_{11}b_{s}(p)\rightarrow0
\tag{39}%
\end{equation}
as $s\rightarrow\infty$.

Now we estimate the $L_{2}^{m}(-p,p)$ norm of the second integral in (38). The
estimates of the third and fourth integrals in (38) are the same. Using the
substitution $t=he^{i\varphi}$ and repeating the proof of (39) we obtain
\[
\int\limits_{(-p,p)}\left\vert \int\limits_{\gamma_{+}(0,h)}G_{s}%
^{+}(t,x)dt\right\vert ^{2}dx=\int\limits_{(-p,p)}\left\vert \int
\limits_{(0,\pi)}G_{s}^{+}(he^{i\varphi},x)ihe^{i\varphi}d\varphi\right\vert
^{2}dx\rightarrow0
\]
as $s\rightarrow\infty.$ Since $R_{s}^{+}(x)$ is the sum of the integrals in
(38), relation (34) and hence equality (17) is proved. In the same way we
prove (18). The theorem is proved.\bigskip
\end{proof}

\begin{remark}
Note that the relation (39) and the last arguments imply the term by term
integration of the series (10a) over $r(h,\varepsilon(h))$, $\gamma
_{+}(0,h),\gamma_{+}(\pi,h)$ and $\gamma_{+}(t_{j},\varepsilon(h))$ for
$j=1,2,...,v(h).$ Similarly, the series
\begin{equation}
f_{t}^{-}(x)=\sum_{k=1}^{K(h)}(f_{t}^{-},X_{k,t})\Psi_{k,t}(x)+\sum
_{\left\vert k\right\vert \geq N(h)}^{{}}\sum_{j=1}^{m}(f_{t}^{-}%
,X_{k,j,t})\Psi_{k,j,t}(x) \tag{10b}%
\end{equation}
can be integrated term by term over $r(h,\varepsilon(h))$, $\gamma
_{-}(0,h),\gamma_{-}(\pi,h)$ and $\gamma_{-}(t_{j},\varepsilon(h))$ for
$j=1,2,...,v(h),$ where $\gamma_{-}=\{\lambda:\overline{\lambda}\in\gamma
_{+}\}.$
\end{remark}

Now to prove Theorem 1 we use the following consequence of Theorem 2.

\begin{corollary}
\bigskip If $f\in W,$ then
\begin{equation}
\int\limits_{r(h,\varepsilon(h))}f_{t}dt=\sum_{k=1}^{K(h)}\int
\limits_{r(h,\varepsilon(h))}(f_{t},X_{k,t})\Psi_{k,t}dt+\sum_{\left\vert
k\right\vert \geq N(h)}^{{}}\sum_{j=1}^{m}\int\limits_{r(h,\varepsilon
(h))}(f_{t},X_{k,j,t})\Psi_{k,j,t}dt, \tag{40}%
\end{equation}
where the series in (40) converges in the $L_{2}^{m}(a,b)$ norm for any
$a,b\in\mathbb{R}.$
\end{corollary}

\begin{proof}
\bigskip Since (39) implies the term by term integration of the series (10a)
over $r(h,\varepsilon(h))$, it follows from (10a) that
\[
\int\limits_{r(h,\varepsilon)}f_{t}^{+}dt=\sum_{k=1}^{K(h)}\int
\limits_{r(h,\varepsilon)}(f_{t}^{+},X_{k,t})\Psi_{k,t}dt+\sum_{\left\vert
k\right\vert \geq N(h)}^{{}}\sum_{j=1}^{m}\int\limits_{r(h,\varepsilon)}%
(f_{t}^{+},X_{k,j,t})\Psi_{k,j,t}dt.
\]
In the same way from (10b) we obtain
\[
\int\limits_{r(h,\varepsilon)}f_{t}^{-}dt=\sum_{k=1}^{K(h)}\int
\limits_{r(h,\varepsilon)}(f_{t}^{-},X_{k,t})\Psi_{k,t}dt+\sum_{\left\vert
k\right\vert \geq N(h)}^{{}}\sum_{j=1}^{m}\int\limits_{r(h,\varepsilon)}%
(f_{t}^{-},X_{k,j,t})\Psi_{k,j,t}dt.
\]
The last two relations with the equality $f_{t}^{+}+f_{t}^{-}=f_{t}^{{}}$ (see
Definition 1) imply the proof of the Corollary.
\end{proof}

Now we are ready to prove Theorem 1.

\textbf{The proof of Theorem 1. }By Definition 1 if $f\in W,$ then there exist
$M$ such that
\[
\left\vert f_{t}(x)\right\vert <M
\]
for all $x\in(-p,p)$ and $t\in\lbrack0,2\pi].$ On the other hand it follows
from the definition of $r(h,\varepsilon(h))$ that the measure of the set
$[0,2\pi]\backslash r(h,\varepsilon(h))$ is less than $6h.$ Therefore, we
have
\[
\left\Vert
{\displaystyle\int\limits_{\lbrack0,2\pi]\backslash r(h,\varepsilon(h))}}
f_{t}dt\right\Vert _{(-p,p)}\leq6Mh.
\]
This with (8) and (40) gives \
\[
\left\Vert f-\frac{1}{2\pi}%
{\displaystyle\int\limits_{r(h,\varepsilon(h))}}
f_{t}dt\right\Vert _{(-p,p)}\leq6Mh
\]
and \
\[
\left\Vert f-\tfrac{1}{2\pi}\left(  \sum_{k=1}^{K(h)}\int
\limits_{r(h,\varepsilon)}(f_{t},X_{k,t})\Psi_{k,t}+\sum_{\left\vert
k\right\vert \geq N(h)}^{{}}\sum_{j=1}^{m}\int\limits_{r(h,\varepsilon)}%
(f_{t},X_{k,j,t})\Psi_{k,j,t}\right)  \right\Vert _{(-p,p)}\leq6Mh.
\]
Now taking $h<\frac{\delta}{6M}$ we obtain the proof of Theorem 1.

\section{Spectral Expansion with Brackets}

Now to obtain a spectral expansion from (17) and (18) we need to change the
curves $l_{+}(\varepsilon,h)$ and $l_{-}(\varepsilon,h)$ to $(-h,2\pi-h).$ In
the other words, we need to change the semicircles $\ $ $\gamma_{\pm}(0,h)$,
$\gamma_{\pm}(\pi,h)$ and $\gamma_{\pm}(t_{j},\varepsilon(h))$ to $(-h,h),$
$(\pi-h,\pi+h)$ and $(t_{j}-\varepsilon(h),t_{j}-\varepsilon(h))$ for
$j=1,2,...,v(h)$ respectively. This is done as follows. The equality (8) can
be written in the form%
\begin{equation}
f(x)=\frac{1}{2\pi}\left(  \int\limits_{r(h,\varepsilon(h))}f_{t}dt+\left(
\sum\limits_{j=1}^{v(h)}\int\limits_{(t_{j}-\varepsilon,t_{j}+\varepsilon
)}f_{t}(x)dt\right)  +\int\limits_{(-h,h)}f_{t}dt+\int\limits_{(\pi-h,\pi
+h)}f_{t}dt\right)  , \tag{41}%
\end{equation}
since
\[
\left(  \cup_{j=1}^{v(h)}(t_{j}-\varepsilon,t_{j}+\varepsilon)\right)  \cup
r(h,\varepsilon(h))\cup(-h,h)\cup(\pi-h,\pi+h)=(-h,2\pi-h].
\]
The first summand of the right hand side of (41) is decomposed by the Bloch
functions due to (40). We say that (40) is the spectral expansion for the
first summand of the right hand side of (41). In this section we obtain the
spectral expansion for the other summand of (41). As is noted in introduction
the intervals $(t_{j}-\varepsilon,t_{j}+\varepsilon)$ and the disks enclosed
by $\gamma_{+}(t_{j},\varepsilon)$ and $\gamma_{-}(t_{j},\varepsilon)$ for
$j=1,2,...,v(h)$ contain only one point $t_{j}$ from $A$ and hence they may
contain at most$\ $one SQ, while the intervals $(-h,h)$ and $(\pi-h,\pi+h)$
and the disks enclosed by $\gamma_{+}(0,h)\cup\gamma_{-}(0,h)$ and $\gamma
_{+}(\pi,h)\cup\gamma_{-}(\pi,h)$ contain, in general, infinitely many SQ.
Therefore replacing $\gamma_{\pm}(t_{j},\varepsilon)$ by $(t_{j}%
-\varepsilon,t_{j}-\varepsilon)$ for $j=1,2,...,v$ is easier than replacing
$\gamma_{\pm}(0,h)$ and $\gamma_{\pm}(\pi,h)$ by $(-h,h)$ and $(\pi-h,\pi+h)$
respectively. In other words, constructing a spectral expansion for the second
term is easier than doing the same for the third and fourth terms in (41).

First we consider the simpler case, that is, investigate the integrals over
$\gamma_{\pm}(t_{j},\varepsilon)$ for $j=1,2,...,v.$ Recall that (see
introduction) $\left\{  t_{1},t_{2},...,t_{v(h)}\right\}  $ is the set of the
quasimomenta $t\in B(h)$, where $B(h)=[h,\pi-h]\cup\lbrack\pi+h,2\pi-h],$ for
which the operator $L_{t}$ has a multiple eigenvalue. Moreover, it follows
from Summary 1$(b)$ that the operator $L_{t_{j}}$ has at most finite number
multiple eigenvalues. Let $\Lambda_{1}(t_{j}),$ $\Lambda_{2}(t_{j}%
),...,\Lambda_{s_{j}}(t_{j})$ be the different multiple eigenvalues of the
operator $L_{t_{j}}.$ Introduce the set
\begin{equation}
\mathbb{T}(s,j):=\left\{  \mid k\mid<N(h):\lambda_{k}(t_{j})=\Lambda_{s}%
(t_{j})\right\}  . \tag{42}%
\end{equation}
Note that by Summary 1$(b)$ if $\mid k\mid\geq N(h),$ then $\lambda
_{k,j}(t_{j})$ is a simple eigenvalue of $L_{t_{j}}.$Therefore the set of all
eigenvalues of $L_{t_{j}}$ coinciding with $\Lambda_{v}(t_{j})$ is the the set
$\lambda_{k}(t_{j})$ for $k\in\mathbb{T}(v,j).$

\begin{theorem}
If $f\in W,$ then for each $x\in(-\infty,\infty)$ the expression%
\begin{equation}
\sum_{k\in\mathbb{T}(s,j)}(f_{t}^{+},X_{k,t})\Psi_{k,t}(x) \tag{43}%
\end{equation}
is integrable over $(t_{j}-\varepsilon,t_{j}+\varepsilon)$ and
\begin{equation}
\sum_{k\in\mathbb{T}(s,j)}\int\limits_{\gamma_{+}(t_{j},\varepsilon)}%
(f_{t}^{+},X_{k,t})\Psi_{k,t}(x)dt=\int\limits_{(t_{j}-\varepsilon
,t_{j}+\varepsilon)}\sum_{k\in\mathbb{T}(s,j)}(f_{t}^{+},X_{k,t})\Psi
_{k,t}(x)dt \tag{44}%
\end{equation}
for all $j=1,2,...,v.$ Theorem remains valid if $f_{t}^{+}$ and $\gamma
_{+}(t_{j},\varepsilon)$ are replaced by $f_{t}^{-}$ and $\gamma_{-}%
(t_{j},\varepsilon)$ respectively.
\end{theorem}

\begin{proof}
The eigenvalues of the operator $L_{t}$ are the roots of the characteristic
determinant
\[
\Delta(\lambda,t)=\det(Y_{j}^{(\nu-1)}(1,\lambda)-e^{it}Y_{j}^{(\nu
-1)}(0,\lambda))_{j,\nu=1}^{n}=
\]%
\[
e^{inmt}+f_{1}(\lambda)e^{i(nm-1)t}+f_{2}(\lambda)e^{i(nm-2)t}+...+f_{nm-1}%
(\lambda)e^{it}+1,
\]
where $Y_{1}(x,\lambda),Y_{2}(x,\lambda),\ldots,Y_{n}(x,\lambda)$ are the
solutions of the matrix equation
\[
Y^{(n)}(x)+P_{2}\left(  x\right)  Y^{(n-2)}(x)+P_{3}\left(  x\right)
Y^{(n-3)}(x)+...+P_{n}(x)Y=\lambda Y(x)
\]
satisfying $Y_{k}^{(j)}(0,\lambda)=0_{m}$ for $j\neq k-1$ and $Y_{k}%
^{(k-1)}(0,\lambda)=I_{m},$ $0_{m}$ and $I_{m}$ are $m\times m$ zero and
identity matrices respectively (see [7, Chap. 3]). It is clear that
$\Delta(\lambda,t)$ is a polynomial of $e^{it}$\ with entire coefficients
$f_{1}(\lambda),f_{2}(\lambda),...$. Therefore the multiple eigenvalues of the
operators $L_{t}$ are the zeros of the resultant $R(\lambda)\equiv
R(\Delta,\Delta^{^{\prime}})$ of the polynomials $\Delta(\lambda,t)$ and
$\Delta^{^{\prime}}(\lambda,t):=\frac{\partial}{\partial\lambda}\Delta
(\lambda,t).$ Since $R(\lambda)$ is an entire function and the large
eigenvalues of $L_{t}$ for $t\neq0,\pi$ are simple (see Summary 1 $(b)$),
either
\[
\ker R=\{\lambda:R(\lambda)=0\}=\{a_{1},a_{2},...\},\text{ }\lim
_{k\rightarrow\infty}\mid a_{k}\mid=\infty
\]
or $\ker R$ is a finite set. If $\Lambda_{v}(t_{j})$ is the $p$-multiple
eigenvalue of $L_{t_{j}},$ then it is the $p$-multiple root of the equation
$\Delta(\lambda,t_{j})=0.$ Therefore using the implicit function theorem one
can easily conclude that there exist the positive constant $\varepsilon$ and
$r$ such that the operators $L_{t}$ for $t\in\{t\in\mathbb{C}:0<\left\vert
t-t_{j}\right\vert <\varepsilon\}$ has only $p$ eigenvalues inside the circle
$C(t_{j})=\left\{  \lambda\in\mathbb{C}:\left\vert \lambda-\Lambda_{v}%
(t_{j})\right\vert =r\right\}  .$ These eigenvalues are simple and are
$\lambda_{k}(t)$ for $k\in\mathbb{T}(v,j).$ Moreover $C(t_{j})$ lies in the
resolvent set of the operators $L_{t}$ for all $t\in\{t\in\mathbb{C}%
:0<\left\vert t-t_{j}\right\vert <\varepsilon\}.$ It is well-known [7, Chap.
3] that the total projection defined by integration of the resolvent of
$L_{t}$ over $C(t_{j})$ has the form
\[
T(x,t):=\int\nolimits_{C(t_{j})}A(x,\lambda,t)d\lambda,
\]
where
\[
A(x,\lambda,t):=\int\nolimits_{0}^{1}G(x,\xi,\lambda,t)f_{t}(\xi)d\xi
\]
and $G(x,\xi,\lambda,t)$ is the Green function of the operator $L_{t}$ defined
by formula (8) of [7, p.117]). Therefore using this formula instead of formula
(34) of [7, page 37] for the Creen function in the scalar case and repeating
the proof of Theorem 2.1 of [12] we get the proof of the theorem.
\end{proof}

By Theorem 3 if $\lambda_{k}(t_{j})$ is a simple eigenvalue then $(f_{t}^{\pm
},X_{k,t})\Psi_{k,t}(x)$ is integrable over $(t_{j}-\varepsilon,t_{j}%
+\varepsilon)$ and
\begin{equation}
\int\limits_{\gamma_{\pm}(t_{j},\varepsilon)}a_{k}^{\pm}(t)\Psi_{k,t}%
(x)dt=\int\limits_{(t_{j}-\varepsilon,t_{j}+\varepsilon)}a_{k}^{\pm}%
(t)\Psi_{k,t}(x)dt, \tag{45}%
\end{equation}
where
\begin{equation}
a_{k}^{\pm}(t)\Psi_{k,t}=(f_{t}^{\pm},X_{k,t})\Psi_{k,t}=\frac{1}{\alpha
_{n}(t)}(f_{t}^{\pm},\Psi_{k,t}^{\ast})\Psi_{k,t}. \tag{46}%
\end{equation}

Now we consider the case of multiple eigenvalue $\Lambda_{s}(t_{j})$. Let
$\mathbb{B}(s,j)$ and $\mathbb{S}(s,j)$ be respectively the subset of
$\mathbb{T}(s,j)$ such that for $k\in\mathbb{B}(s,j)$ and $k\in\mathbb{S}%
(s,j)$ the function $\frac{1}{\alpha_{k}}$ is integrable and nonintegrable in
$(t_{j}-\varepsilon,t_{j}+\varepsilon)$. Then we have the following theorem.

\begin{theorem}
Let $\Lambda_{s}(t_{j})$ be a multiple eigenvalue and $f\in W.$ Then
$a_{k}^{\pm}(t)\Psi_{k,t}(x)$ for $k\in\mathbb{B}(s,j)$ and
\begin{equation}
S_{s,j}^{\pm}(x,t)=\sum_{k\in\mathbb{S}(s,j)}a_{k}^{\pm}(t)\Psi_{k,t}(x)
\tag{47}%
\end{equation}
are integrable with respect to $t$ over $(t_{j}-\varepsilon,t_{j}%
+\varepsilon)$ and
\begin{equation}
\sum_{k\in\mathbb{T}(s,j)}\int\limits_{\gamma_{\pm}(t_{j},\varepsilon)}%
a_{k}^{\pm}(t)\Psi_{k,t}(x)dt=F^{\pm}(s,j,x)+\sum_{k\in\mathbb{B}(s,j)}%
\int\limits_{(t_{j}-\varepsilon,t_{j}+\varepsilon)}a_{k}^{\pm}(t)\Psi
_{k,t}(x)dt, \tag{48}%
\end{equation}
where
\begin{equation}
F^{\pm}(s,j,x)=\int\limits_{(t_{j}-\varepsilon,t_{j}+\varepsilon)}\sum
_{k\in\mathbb{S}(s,j)}a_{k}^{\pm}(t)\Psi_{k,t}(x)dt=\lim_{\delta\rightarrow
0}\sum_{k\in\mathbb{S}(s,j)}\int\limits_{\delta<\left\vert t-t_{j}\right\vert
\leq\varepsilon}a_{k}^{\pm}(t)\Psi_{k,t}(x)dt. \tag{49}%
\end{equation}

If the multiple eigenvalue $\Lambda_{s}(t_{j})$ is not an ESS, then
\begin{equation}
\sum_{k\in\mathbb{T}(s,j)}\int\limits_{\gamma(t_{j},\varepsilon)}a_{k}%
^{+}(t)\Psi_{k,t}(x)dt=\sum_{k\in\mathbb{T}(s,j)}\int\limits_{(t_{j}%
-\varepsilon,t_{j}+\varepsilon)}a_{k}^{+}(t)\Psi_{k,t}(x)dt. \tag{50}%
\end{equation}

\end{theorem}

\begin{proof}
If $k\in\mathbb{B}(s,j),$ then by the definition of $\mathbb{B}(s,j),$ the
function $\frac{1}{\alpha_{n}}$ is integrable over $(t_{j}-\varepsilon
,t_{j}+\varepsilon).$ On the other hand, using (46) Definition 1 and taking
into account that $\Psi_{k,t}^{\ast}$ is the normalized eigenfunction we
obtain
\[
\left\vert a_{k}^{\pm}(t)\Psi_{k,t}(x)\right\vert \leq\frac{1}{\left\vert
\alpha_{n}(t)\right\vert }\left\Vert f_{t}^{\pm}\right\Vert \left\vert
\Psi_{k,t}(x)\right\vert \leq c\frac{1}{\left\vert \alpha_{n}(t)\right\vert
}\left\vert \Psi_{k,t}(x)\right\vert .
\]
Moreover, in the following lemma (Lemma 2) we prove that%
\begin{equation}
\left\vert \Psi_{k,t}(x)\right\vert <c_{11} \tag{51}%
\end{equation}
for all $t\in\left(  (t_{j}-\varepsilon,t_{j})\cup(t_{j},t_{j}+\varepsilon
)\right)  $ and $x\in\lbrack0,1].$ Therefore
\[
\left\vert a_{k}^{\pm}(t)\Psi_{k,t}(x)\right\vert \leq c_{12}\frac
{1}{\left\vert \alpha_{n}(t)\right\vert }.
\]
This inequality implies that $a_{k}^{\pm}(t)\Psi_{k,t}(x)$ is integrable on
$(t_{j}-\varepsilon,t_{j}+\varepsilon)$ for all\ $k\in\mathbb{B}(s,j).$
Therefore by Theorem 3 $S_{s,j}^{\pm}(x,t)$ is also integrable on
$(t_{j}-\varepsilon,t_{j}+\varepsilon)$ and (48) holds. Equality (49) follows
from the absolutely continuity of the Lebesque integral.

If the multiple eigenvalue$\Lambda_{s}(t_{j})$ is not ESS, then it follows
\ from Definition 2 that $\frac{1}{\alpha_{k}}$ integrable on $(t_{j}%
-\varepsilon,t_{j}+\varepsilon)$ for all\ $k\in\mathbb{T}(s,j).$ In other
words, $\mathbb{S}(s,j)$ is an empty set and $\mathbb{B}(s,j)=\mathbb{T}%
(s,j)$. Therefore equality (50) follows from (48).
\end{proof}

Now we prove (51) by using the following well-known formula
\begin{equation}
\left(  \lambda_{k}(t)-\left(  2\pi pi+ti\right)  ^{n}\right)  \left(
\Psi_{k,t},\varphi_{p,s,t}\right)  =\left(  \sum\limits_{\nu=2}^{n}P_{\nu}%
\Psi_{k,t}^{(n-\nu)},\varphi_{p,s,t}\right)  , \tag{52}%
\end{equation}
where $\varphi_{p,s,t}(x)=e_{s}e^{i\left(  2\pi p+t\right)  x}$ and $\left\{
e_{s}:s=1,2,...,m\right\}  $ is a standard basis of $\mathbb{C}^{m}.$

\begin{lemma}
If $\mid k\mid<N(h),$ then (51) holds for all $t\in\left(  (t_{j}%
-\varepsilon,t_{j})\cup(t_{j},t_{j}+\varepsilon)\right)  $ and $x\in
\lbrack0,1].$
\end{lemma}

\begin{proof}
It is clear that there exists $l\in\mathbb{N}$ such that
\[
|\lambda_{k}(t)-\left(  2\pi pi+it\right)  ^{n}|>|p|^{n}%
\]
for all $|p|\geq l$ and $t\in\left(  (t_{j}-\varepsilon,t_{j})\cup(t_{j}%
,t_{j}+\varepsilon)\right)  .$ On the other hand, using integration by parts
$n-v$ times and then the inclusion $p_{v,i,j}^{(n-v)}\in L_{2}[0,1]$ we
obtain
\[
\left\vert \left(  P_{v}\Psi_{k,t}^{(n-\nu)},\varphi_{p,s,t}\right)
\right\vert =\left\vert \left(  \Psi_{k,t}^{(n-\nu)},P_{v}^{\ast}%
\varphi_{p,s,t}\right)  \right\vert =\left\vert \left(  \Psi_{k,t},\left(
P_{2}^{\ast}\varphi_{p,s,t}\right)  ^{(n-v)}\right)  \right\vert \leq
\]%
\[
\left\Vert \left(  P_{2}^{\ast}\varphi_{p,s,t}\right)  ^{(n-v)}\right\Vert
<\left(  c_{13}\left\vert p\right\vert ^{n-v}\right)  .
\]
Therefore, by (52) we have
\[
\left\vert \left(  \Psi_{k,t},\varphi_{p,s,t}\right)  \right\vert
<\frac{c_{14}}{p^{2}}%
\]
for all $|p|\geq l$ and $t\in\left(  (t_{j}-\varepsilon,t_{j})\cup(t_{j}%
,t_{j}+\varepsilon)\right)  .$ This inequality with the Fourier decomposition
of $\Psi_{k,t}$ with respect to the basis $\left\{  \varphi_{p,s,t}%
:p\in\mathbb{Z},\text{ }s=1,2,...,m\right\}  $ yields (51).
\end{proof}

\begin{notation}
The set of the ESS is the subset of the set of the multiple eigenvalues and
equality (50) shows that we need consider the multiple eigenvalues of
$L_{t_{j}}$ which are the ESS. Therefore, for simplicity of notation and
without loss of generality, we assume that $\left\{  t_{1},t_{2}%
,...,t_{v(h)}\right\}  $ is the set of SQ lying in $[h,\pi-h]\cup\lbrack
\pi+h,2\pi-h].$ For each $j=1,2,...,v(h)$ let $\left\{  \Lambda_{s}%
(t_{j}):\text{ }s=1,2,...,s_{j}\right\}  $ be the set of ESS.
\end{notation}

If $\Lambda_{s}(t_{j})$ is an ESS, then some elements of the set
\begin{equation}
\left\{  a_{k}(t)\Psi_{k,t}(x):k\in\mathbb{T}(s,j)\right\}  \tag{53}%
\end{equation}
may become nonintegrable on $(t_{j}-\varepsilon,t_{j}+\varepsilon),$ while
some of elements of (53) may be integrable and the total sum of elements of
(53) is integrable due to the cancellations of the nonintegrable terms, where
\[
a_{k}(t)=(f_{t},X_{k,t})=\frac{1}{\alpha_{k}(t)}(f_{t},\Psi_{k,t}^{\ast}).
\]
In fact, we compound together only the nonintegrable elements of (53). Using
this argument we obtain the spectral expansion for the second summand of (41).

\begin{theorem}
If $f\in W,$ then
\begin{equation}
\int\limits_{(t_{j}-\varepsilon,t_{j}+\varepsilon)}f_{t}(x)dt=\sum
_{k\in\mathbb{B}(j)}\int\limits_{(t_{j}-\varepsilon,t_{j}+\varepsilon)}%
a_{k}(t)\Psi_{k,t}dt+\sum_{\left\vert k\right\vert \geq N(h)}^{{}}\sum
_{j=1}^{m}\int\limits_{(t_{j}-\varepsilon,t_{j}+\varepsilon)}a_{k,j}%
(t)\Psi_{k,j,t}dt+ \tag{54}%
\end{equation}%
\[
\sum_{s=1}^{s_{j}}\int\limits_{(t_{j}-\varepsilon,t_{j}+\varepsilon)}\left[
\sum_{k\in\mathbb{S}(s,j)}a_{k}(t)\Psi_{k,t}\right]  dt,
\]
where $\mathbb{B}(j)=\left\{  1,2,...,K(h)\right\}  \backslash\left(
{\textstyle\bigcup\nolimits_{s=1}^{s_{j}}}
\mathbb{S}(s,j)\right)  ,$
\begin{equation}
\int\limits_{(t_{j}-\varepsilon,t_{j}+\varepsilon)}\sum_{k\in\mathbb{S}%
(s,j)}a_{k}(t)\Psi_{k,t}dt=\lim_{\delta\rightarrow0}\sum_{k\in\mathbb{S}%
(s,j)}\int\limits_{\delta<\left\vert t-t_{j}\right\vert \leq\varepsilon}%
a_{k}(t)\Psi_{k,t}(x)dt \tag{55}%
\end{equation}
and the series in (54) converges in the $L_{2}^{m}(a,b)$ norm for any
$a,b\in\mathbb{R}$.
\end{theorem}

\begin{proof}
By Definition 1 we have
\begin{equation}
\int\limits_{(t_{j}-\varepsilon,t_{j}+\varepsilon)}f_{t}(x)dt=\int
\limits_{\gamma_{+}(t_{j},\varepsilon)}f_{t}^{+}(x)dt+\int\limits_{\gamma
_{-}(t_{j},\varepsilon)}f_{t}^{-}(x)dt. \tag{56}%
\end{equation}
Moreover, it follows from Remark 1 that
\begin{equation}
\int\limits_{\gamma_{+}(t_{j},\varepsilon)}f_{t}^{+}(x)dt=\sum_{k=1}^{K(h)}%
{\displaystyle\int\limits_{\gamma_{+}(t_{j},\varepsilon)}}
(f_{t}^{+},X_{k,t}),\Psi_{k,t}dt+\sum_{\left\vert k\right\vert \geq N(h)}^{{}%
}\sum_{j=1}^{m}%
{\displaystyle\int\limits_{\gamma_{+}(t_{j},\varepsilon)}}
(f_{t}^{+},X_{k,j,t})\Psi_{k,j,t}dt \tag{57}%
\end{equation}
and
\begin{equation}
\int\limits_{\gamma_{-}(t_{j},\varepsilon)}f_{t}^{-}(x)dt=\sum_{k=1}^{K(h)}%
{\displaystyle\int\limits_{\gamma_{-}(t_{j},\varepsilon)}}
(f_{t}^{-},X_{k,t}),\Psi_{k,t}dt+\sum_{\left\vert k\right\vert \geq N(h)}^{{}%
}\sum_{j=1}^{m}%
{\displaystyle\int\limits_{\gamma_{-}(t_{j},\varepsilon)}}
(f_{t}^{-},X_{k,j,t})\Psi_{k,j,t}dt. \tag{58}%
\end{equation}
Thus, according to (56) the sum of the left-hand sides of (57) and (58) is
equal to the left-hand side of (54). Therefore, to prove (54), it suffices to
show that the sum of the right-hand sides of (57) and (58) is equal to the
right-hand side of (54). First consider the terms of the right-hand sides of
(57) and (58) with indices $\left\vert k\right\vert \geq N(h).$ By Summary
1$(b),$ the eigenvalues $\lambda_{k,j}(t)$ for $\left\vert k\right\vert \geq
N(h)$ and $t\in\gamma_{\pm}(t_{j},\varepsilon)$ are simple. Therefore,
according to (45) we have
\[%
{\displaystyle\int\limits_{\gamma_{\pm}(t_{j},\varepsilon)}}
(f_{t}^{\pm},X_{k,j,t})\Psi_{k,j,t}dt=%
{\displaystyle\int\limits_{(t_{j}-\varepsilon,t_{j}+\varepsilon)}}
(f_{t}^{\pm},X_{k,j,t})\Psi_{k,j,t}dt
\]
and
\begin{equation}%
{\displaystyle\int\limits_{\gamma_{+}(t_{j},\varepsilon)}}
(f_{t}^{+},X_{k,j,t})\Psi_{k,j,t}dt+%
{\displaystyle\int\limits_{\gamma_{-}(t_{j},\varepsilon)}}
(f_{t}^{-},X_{k,j,t})\Psi_{k,j,t}dt=%
{\displaystyle\int\limits_{(t_{j}-\varepsilon,t_{j}+\varepsilon)}}
(f_{t},X_{k,j,t})\Psi_{k,j,t}dt \tag{59}%
\end{equation}
for $\left\vert k\right\vert \geq N(h),$ since $f_{t}^{+}+f_{t}^{-}=f_{t}.$
Thus, to prove (54) it remains to show that
\begin{equation}
\sum_{k=1}^{K(h)}%
{\displaystyle\int\limits_{\gamma_{+}(t_{j},\varepsilon)}}
(f_{t}^{+},X_{k,t}),\Psi_{k,t}dt+\sum_{k=1}^{K(h)}%
{\displaystyle\int\limits_{\gamma_{-}(t_{j},\varepsilon)}}
(f_{t}^{-},X_{k,t}),\Psi_{k,t}dt= \tag{60}%
\end{equation}%
\[
\sum_{s=1}^{s_{j}}\int\limits_{(t_{j}-\varepsilon,t_{j}+\varepsilon)}\left(
\sum_{k\in\mathbb{S}(s,j)}a_{k}(t)\Psi_{k,t}\right)  dt+\sum_{k\in
\mathbb{B}(j)}\int\limits_{(t_{j}-\varepsilon,t_{j}+\varepsilon)}a_{k}%
(t)\Psi_{k,t}dt.
\]

The set $\left\{  1,2,...,K(h)\right\}  $ is the union of
\[
\mathbb{K}(j)=:\left\{  k:1\leq k<K(h),\lambda_{k}(t_{j})\text{ is a simple
eigenvalue}\right\}
\]
and $%
{\textstyle\bigcup\nolimits_{s=1}^{s_{j}}}
\mathbb{T}(s,j).$

If $k\in\mathbb{K}(j)$, then arguing as in the case $\left\vert k\right\vert
\geq N(h)$ (see above) we obtain
\begin{equation}
\sum_{k\in\mathbb{K}(j)}%
{\displaystyle\int\limits_{\gamma_{+}(t_{j},\varepsilon)}}
(f_{t}^{+},X_{k,t}),\Psi_{k,t}dt+\sum_{k\in\mathbb{K}(j)}%
{\displaystyle\int\limits_{\gamma_{-}(t_{j},\varepsilon)}}
(f_{t}^{-},X_{k,t}),\Psi_{k,t}dt=\sum_{k\in\mathbb{K}(j)}\int\limits_{(t_{j}%
-\varepsilon,t_{j}+\varepsilon)}a_{k}(t)\Psi_{k,t}dt. \tag{61}%
\end{equation}
If $k\in\mathbb{T}(s,j),$ then using ( 48) we get
\begin{equation}
\sum_{k\in\mathbb{T}(s,j)}\int\limits_{\gamma_{+}(t_{j},\varepsilon)}a_{k}%
^{+}(t)\Psi_{k,t}(x)dt+\sum_{k\in\mathbb{T}(s,j)}\int\limits_{\gamma_{-}%
(t_{j},\varepsilon)}a_{k}^{-}(t)\Psi_{k,t}(x)dt= \tag{62}%
\end{equation}%
\[
F(s,j,x)+\sum_{k\in\mathbb{B}(s,j)}\int\limits_{(t_{j}-\varepsilon
,t_{j}+\varepsilon)}a_{k}(t)\Psi_{k,t}(x)dt,
\]
where%
\begin{equation}
F(s,j,x)=\int\limits_{(t_{j}-\varepsilon,t_{j}+\varepsilon)}\sum
_{k\in\mathbb{S}(s,j)}a_{k}(t)\Psi_{k,t}dt. \tag{63}%
\end{equation}
Thus (61) and (62) imply (60) which completes the prove of (54). \ Note that
(55) follows from the absolute continuity of the Lebesque integral.
\end{proof}

\ Now we consider the integral over $(-h,h)$ in (41). The consideration of
integral over $(\pi-h,\pi+h)$ is similar. The complexity of the investigations
of the integral over $(-h,h)$ is the following. In general, in the domain
enclosed by $\gamma(0,h)$ and $(-h,h)$ there exist infinite number of points
of $A$ and the interval $(-h,h)$ may contain, infinite number of SQ defined in
Definition 2. That is why, the set of indices $k$ for which $\frac{1}%
{\alpha_{k}}$ is nonintegrable over $(-h,h)$ is not finite in the general case
and may coincide with the set of all indices. In other words, all terms of the
series (10), (10a) and (10b) may become nonintegrable on $(-h,h).$ This
situation complicates the replacement of $\gamma(0,h)$ by $[-h,h]$.

Thus, we need to consider in detail the intervals $(-h,h)$ and $(\pi-h,\pi
+h)$. First, we prove that the disks $\{\lambda\in\mathbb{C}:\left\vert
\lambda-(2\pi ki)^{n}\right\vert <\left\vert k\right\vert ^{n-1}\}$ and
$\{\lambda\in\mathbb{C}:\left\vert \lambda-(2\pi ki+i\pi)^{n}\right\vert
<\left\vert k\right\vert ^{n-1}\}$ for the large positive value of $k$
contains $2m$ eigenvalues (counting the multiplicity) of the operator $L_{t}$
for $\left\vert t\right\vert \leq h$ and $\left\vert t-\pi\right\vert \leq h$
respectively. To do this, we use the following formula
\begin{equation}
\left(  \lambda(t)-\left(  2\pi ki+ti\right)  ^{n}\right)  \left(
\Psi_{\lambda(t)},\varphi_{k,s,t}^{\ast}\right)  =\left(  z\sum\limits_{\nu
=2}^{n}P_{\nu}\Psi_{\lambda(t)}^{(n-\nu)},\varphi_{k,s,t}^{\ast}\right)  ,
\tag{64}%
\end{equation}
where $\Psi_{\lambda(t)}$ is a normalized eigenfunction of $L_{t,z}$
corresponding to the eigenvalue $\lambda(t)$,
\[
L_{t,z}=L_{t}(0)+z(L_{t}-L_{t}(0)),\text{ }0\leq z\leq1,
\]
$L_{t}(0)$ denotes the case when all coefficients of (1) are zero,
\begin{equation}
\varphi_{k,s,t}^{\ast}(x)=e_{s}e^{i\left(  2\pi k+\overline{t}\right)  x}
\tag{65}%
\end{equation}
and $\left\{  e_{s}:s=1,2,...,m\right\}  $ is the standard basis of
$\mathbb{C}^{m}.$ Formula (64) can be obtained from $L_{t,z}\Psi_{\lambda
(t)}=\lambda(t)\Psi_{\lambda(t)}$ by multiplying by $\varphi_{k,s,t}^{\ast
}(x)$.

\begin{theorem}
There exists a positive integer $N(0,h)$ such that all eigenvalues of the
operator $L_{t,z}$ lie in the union of the disks
\begin{equation}
\{\lambda\in\mathbb{C}:\left\vert \lambda-(2\pi ki)^{n}\right\vert <\left(
N(0,h)\right)  ^{n-1}\} \tag{66}%
\end{equation}
for $0\leq k<N(0,h)$ and
\begin{equation}
\{\lambda\in\mathbb{C}:\left\vert \lambda-(2\pi ki)^{n}\right\vert <k^{n-1}\}
\tag{67}%
\end{equation}
for $k\geq N(0,h),$ where $\left\vert t\right\vert \leq h<\frac{1}{15\pi}$,
$z\in\lbrack0,1]$.
\end{theorem}

\begin{proof}
Suppose the contrary, that there exists an eigenvalue $\lambda(t)$ lying
outside all disks (66) and (67). Using Parseval's equality for the orthonormal
basis $\left\{  e_{s}e^{i2\pi kx}:k\in\mathbb{Z},\text{ }s=1,2,...,m\right\}
$ and (64), we obtain
\begin{equation}
\left\Vert e^{-itx}\Psi_{\lambda(t)}\right\Vert =\sum_{\substack{k\in
\mathbb{Z},\text{ }\\s=1,2,...,m}}\left\vert \left(  \Psi_{\lambda(t)}%
,\varphi_{k,s,t}^{\ast}\right)  \right\vert ^{2}\leq\sum_{\substack{k\in
\mathbb{Z},\text{ }\\s=1,2,...,m}}\frac{\left\vert \left(  z\sum
\limits_{\nu=2}^{n}P_{\nu}\Psi_{\lambda(t)}^{(n-\nu)},\varphi_{k,s,t}^{\ast
}\right)  \right\vert ^{2}}{\left\vert \lambda(t)-(2\pi ki+it)^{n}\right\vert
^{2}}. \tag{68}%
\end{equation}
Now, in order to get a contradiction, we prove that the right-hand side of
(68) is small for a large value of $N(0,h)$.

Consider the cases: $\left\vert k\right\vert \geq N(0,h)$ and $\left\vert
k\right\vert <N(0,h)$ separately. First, consider the case $\left\vert
k\right\vert \geq N(0,h).$ Since $\lambda(t)$ lies outside all the disks (67),
the inequality%
\[
\left\vert \lambda(t)-(2k\pi i+it)^{n}\right\vert \geq\frac{2}{3}\left\vert
k\right\vert ^{n-1}%
\]
holds for all $\left\vert k\right\vert \geq N(0,h)$ and $\left\vert
t\right\vert <\frac{1}{15\pi}.$ On the other hand, using integration by parts
$n-v$ times and then the inclusion $p_{v,i,j}^{(n-v)}\in L_{2}[0,1]$ we
obtain
\begin{equation}
\left\vert \left(  P_{v}\Psi_{\lambda(t)}^{(n-\nu)},\varphi_{k,s,t}^{\ast
}\right)  \right\vert ^{2}=\left\vert \left(  \Psi_{\lambda(t)}^{(n-\nu
)},P_{v}^{\ast}\varphi_{k,s,t}^{\ast}\right)  \right\vert ^{2}=\left\vert
\left(  \Psi_{\lambda(t)},\left(  P_{2}^{\ast}\varphi_{k,s,t}^{\ast}\right)
^{(n-v)}\right)  \right\vert ^{2}\leq\tag{69}%
\end{equation}%
\[
\left\Vert \left(  P_{2}^{\ast}\varphi_{k,s,t}^{\ast}\right)  ^{(n-v)}%
\right\Vert ^{2}<\left(  c_{15}\left\vert k\right\vert ^{n-v}\right)  ^{2}.
\]
Therefore, we have
\begin{equation}
\sum_{\substack{\left\vert k\right\vert \geq N(0,h),\text{ }\\s=1,2,...,m}%
}\frac{\left\vert \left(  \sum\limits_{\nu=2}^{n}P_{\nu}\Psi_{\lambda
(t)}^{(n-\nu)},\varphi_{k,s,t}^{\ast}\right)  \right\vert ^{2}}{\left\vert
\lambda(t)-(2k\pi i+it)^{n}\right\vert ^{2}}<c_{16}\sum_{\left\vert
k\right\vert \geq N(0,h)}\frac{1}{\left\vert k\right\vert ^{2}}<\frac{c_{17}%
}{N(0,h)}. \tag{70}%
\end{equation}

Now consider the case $\left\vert k\right\vert <N(0,h).$ Since $\lambda(t)$
lies outside all disks (66) we have
\begin{equation}
\left\vert \lambda(t)-(2k\pi i+it)^{n}\right\vert \geq\frac{2}{3}\left(
N(0,h)\right)  ^{n-1} \tag{71}%
\end{equation}
for $\left\vert k\right\vert <N(0,h)$ and $\left\vert t\right\vert <\frac
{1}{15\pi}.$ Therefore, arguing as in the proof of (69), we get
\begin{equation}
\left\vert \left(  P_{v}\Psi_{\lambda(t)}^{(n-\nu)},\varphi_{k,s,t}^{\ast
}\right)  \right\vert ^{2}<\left(  c_{18}\left(  N(0,h)\right)  ^{n-v}\right)
^{2} \tag{72}%
\end{equation}
for $\left\vert k\right\vert <N(0,h).$ Using (71) and (72) and taking into
account that the number of $k$ satisfying the inequality $\left\vert
k\right\vert <N(0,h)$ is $2N(0,h)-1,$ we obtain
\[
\sum_{\substack{\left\vert k\right\vert <N(0,h),\text{ }\\s=1,2,...,m}%
}\frac{\left\vert \left(  \sum\limits_{\nu=2}^{n}P_{\nu}\Psi_{\lambda
(t)}^{(n-\nu)},\varphi_{k,s,t}^{\ast}\right)  \right\vert ^{2}}{\left\vert
\lambda(t)-(2k+t)^{n}\right\vert ^{2}}<\frac{c_{19}}{N(0,h)}.
\]
This inequality with (70) implies that the right-hand side of (68) is a small
number for the large values of $N(0,h).$ This contradiction provides the proof
of the theorem.
\end{proof}

In the same way we prove the following.

\begin{theorem}
There exists a positive integer $N(\pi,h)$ such that all eigenvalues of the
operator $L_{t,z}$ lie in the union of the disks
\begin{equation}
\{\lambda\in\mathbb{C}:\left\vert \lambda-(i2\pi k+i\pi)^{n}\right\vert
<\left(  N(\pi,h)\right)  ^{n-1}\} \tag{73}%
\end{equation}
for $0\leq k<N(\pi,h)$ and
\begin{equation}
\{\lambda\in\mathbb{C}:\left\vert \lambda-(i2\pi k+i\pi)^{n}\right\vert
<k^{n-1}\} \tag{74}%
\end{equation}
for $k\geq N(\pi,h),$ where $\left\vert t-\pi\right\vert \leq h<\frac{1}%
{15\pi}$, $z\in\lbrack0,1].$
\end{theorem}

Now using Theorems 6 and 7 we obtain the following results.

\begin{theorem}
$(a)$ The disks (67) and (74) for $k\geq N(0,h)$ and $k\geq N(\pi,h)$ contain
$2m$ eigenvalues of the operator $L_{t}$ for $\left\vert t\right\vert \leq h$
and $\left\vert t-\pi\right\vert \leq h$ respectively.

$(b)$ There are closed curves $\Gamma(0)$ and $\Gamma(\pi)$ such that they
enclose all the eigenvalues of $L_{t}$ for $|t|\leq h$ and $\left\vert
t-\pi\right\vert \leq h$ that are absent in (67) and (74) respectively. The
numbers of \ the eigenvalues of $L_{t}$ for $\left\vert t\right\vert \leq h$
and $\left\vert t-\pi\right\vert \leq h$ lying inside $\Gamma(0)$ and
$\Gamma(\pi)$ are equal to $K(0,h)$ and $K(\pi,h)$ respectively, where
$K(0,h)=(2N(0,h)-1)m$ and $K(\pi,h)=2mN(\pi,h).$
\end{theorem}

\begin{proof}
The proof of the theorem is given for the case $\left\vert t\right\vert \leq
h.$ It is the same for the case $\left\vert t-\pi\right\vert \leq h$.

$(a)$ It follows from Theorem 6 that the boundary of the disks (67) lies in
the resolvent set of the operators $L_{t,z}$ for all $z\in\lbrack0,1].$
Therefore, taking into account that the family $L_{t,z}$ is halomorphic with
respect to $z,$ we obtain that the number of eigenvalues (counting the
multiplicity) of the operators $L_{t,0}=L_{t}(0)$ and $L_{t,1}=L_{t}$ lying
inside (67) are the same (see [3, Chap. 7]). Now the proof follows from the
obvious fact that the operator $L_{t,0}$ has only $2m$ eigenvalues lying
inside (67).

$(b)$ One can easily verify that there exists a closed curve $\Gamma(0)$ which
encloses all the disks in (66) and has no intersection points with the disks
of (67). This with the Theorem 6 implies that $\Gamma(0)$ lies in the
resolvent set of $L_{t,z}$ for all $z\in\lbrack0,1]$ and contains all
eigenvalues of $L_{t}$ apart from the eigenvalues lying in (67). It means that
apart from the eigenvalues lying inside (67) there exist $K(0,h)$ eigenvalues
of the operator $L_{t}$ for $0<|t|\leq h,$ since $\Gamma(0)$ contains $K(0,h)$
eigenvalues of the operator $L_{t,0}.$
\end{proof}

\begin{notation}
For $|t|\leq h$ denote the eigenvalues of the operator $L_{t}$ lying in (67)
by
\begin{equation}
\lambda_{k,1}(t),\text{ }\lambda_{k,2}(t),...,\lambda_{k,m}(t),\text{ }%
\lambda_{-k,1}(t),\text{ }\lambda_{-k,2}(t),...,\lambda_{-k,m}(t), \tag{75}%
\end{equation}
where $k\geq N(0,h).$ The eigenvalues of $L_{t}$ lying inside $\Gamma(0)$ are
denoted by $\lambda_{k}(t)$ for \ $k\in\mathbb{N}(0),$ where $\mathbb{N}%
(0)=\left\{  1,2,...,K(0,h\right\}  .$

Similarly, for $|t-\pi|\leq h$ denote the eigenvalues of the operator $L_{t}$
lying in (74) by
\begin{equation}
\lambda_{k,1}(t),\text{ }\lambda_{k,2}(t),...,\lambda_{k,m}(t),\text{ }%
\lambda_{-k-1,1}(t),\text{ }\lambda_{-k-1,2}(t),...,\lambda_{-k-1,m}(t),
\tag{76}%
\end{equation}
where $k\geq N(\pi,h)$. The eigenvalues of $L_{t}$ lying inside $\Gamma(\pi)$
are denoted by $\lambda_{k}(t)$ for $k\in\mathbb{N}(\pi),$ where
$\mathbb{N}(\pi)=\left\{  1,2,...,K(\pi,h\right\}  .$
\end{notation}

\begin{remark}
The eigenvalues \ (75) may become multiple eigenvalues and ESS and hence all
expressions $a_{s,j}(t)\Psi_{s,j,t}$ for $s=k,-k$ and $j=1,2,...,m$ may become
nonintegrable. This situation requires to compound together these expressions.
This is discussed in detail in conclusion Section 4.
\end{remark}

Now we are ready to consider the intervals $[-h,h]$ and $[\pi-h,\pi+h].$
Instead of the circle $C(t_{j})$ used in the proof of Theorem 3 \ taking the
circle%
\[
C(0):=\left\{  z\in\mathbb{C}:\left\vert z-(i2k\pi)^{n}\right\vert
=k^{n-1}\right\}  ,
\]
using Notation 2 and repeating the proof of Theorem 3 we obtain the following result.

\begin{proposition}
For each $k\geq N(0,h)$ the equality
\begin{equation}
\sum_{s=k,-k}^{{}}\sum_{j=1}^{m}\int\limits_{\gamma_{\pm}(0,h)}a_{s,j}^{\pm
}(t)\Psi_{s,j,t}(x)dt=\int\limits_{[-h,h]}\sum_{s=k,-k}^{{}}\sum_{j=1}%
^{m}a_{s,j}^{\pm}(t)\Psi_{s,j,t}(x)dt \tag{77}%
\end{equation}
holds.
\end{proposition}

Similarly, in the case $[\pi-h,\pi+h]$, instead of $C(0)$ taking the circle
\[
C(\pi):=\left\{  z\in\mathbb{C}:\left\vert z-(i(2k\pi+\pi))^{n}\right\vert
=k^{n-1}\right\}
\]
we obtain

\begin{proposition}
For each $k\geq N(\pi,h)$ the following equality holds.
\begin{equation}
\sum_{s=k,-k-1}^{{}}\sum_{j=1}^{m}\int\limits_{\gamma_{\pm}(\pi,h)}%
a_{s,j}^{\pm}(t)\Psi_{s,j,t}(x)dt=\int\limits_{[\pi-h,\pi+h]}\sum
_{s=k,-k-1}^{{}}\sum_{j=1}^{m}a_{s,j}^{\pm}(t)\Psi_{s,j,t}(x)dt. \tag{78}%
\end{equation}

\end{proposition}

Now let us consider the terms $a_{k}^{\pm}(t)\Psi_{k,t}$ for $k\in
\mathbb{N}(0)$ and for $k\in\mathbb{N}(\pi)$. Instead of the curve $C(t_{j})$
using $\Gamma(0)$ and repeating the proof of Theorem 3 we obtain
\begin{equation}
\sum_{k\in\mathbb{N}(0)}\int\limits_{\gamma_{\pm}(0,h)}a_{k}^{\pm}%
(t)\Psi_{k,t}(x)dt=\int\limits_{[-h,h]}\sum_{k\in\mathbb{N}(0)}a_{k}^{\pm
}(t)\Psi_{k,t}(x)dt. \tag{79}%
\end{equation}
In the same way we get
\begin{equation}
\sum_{k\in\mathbb{N}(\pi)}\int\limits_{\gamma_{\pm}(\pi,h)}a_{k}^{\pm}%
(t)\Psi_{k,t}(x)dt=\int\limits_{[\pi-h,\pi+h]}\sum_{k\in\mathbb{N}(\pi)}%
a_{k}^{\pm}(t)\Psi_{k,t}(x)dt. \tag{80}%
\end{equation}

Now we consider the right-hand sides of (79) and (80). To do this recall that
(see the proof of Theorem 3) the set of the multiple Bloch eigenvalues are
either finite or countable set which has no finite limit point. Therefore
$\Gamma(0)$ encloses a finite number multiple eigenvalues $a_{1}%
,a_{2},...a_{s}$. For each $a_{k}$ there are $nm$ values $t_{k,1}%
,t_{k,2},...,t_{k,nm}$ of $t$ satisfying $\Delta(a_{k},t)=0.$ Let $\left\{
t_{v+1},t_{v+2},...,t_{u}\right\}  $ be the set of SQ such that $t_{j}%
\in(-h,h)$ for $j=v+1,v+2,...,u$ and the operator $L_{t_{j}}$ has a multiple
eigenvalue lying inside $\Gamma(0)$. It is clear that
\[
\left\{  t_{v+1},t_{v+2},...,t_{u}\right\}  \subset\left\{  t_{k,1}%
,t_{k,2},...,t_{k,nm}:k=1,2,...,s\right\}
\]
Let $\Lambda_{1}(t_{j}),$ $\Lambda_{2}(t_{j}),...,\Lambda_{u_{j}}(t_{j})$ be
the different multiple eigenvalues of the operator $L_{t_{j}}$ lying inside
$\Gamma(0)$, where $j=v+1,v+2,...,u.$ Introduce the set
\begin{equation}
\mathbb{T}(s,j):=\left\{  k\in\mathbb{N}(0):\lambda_{k}(t_{j})=\Lambda
_{s}(t_{j})\right\}  . \tag{81}%
\end{equation}
Let $\mathbb{B}(s,j)$ and $\mathbb{S}(s,j)$ be respectively the subset of
$\mathbb{T}(s,j)$ such that for $k\in\mathbb{B}(s,j)$ and $k\in\mathbb{S}%
(s,j)$ the function $\frac{1}{\alpha_{k}}$ is integrable and nonintegrable in
$(t_{j}-\varepsilon,t_{j}+\varepsilon)$. Then we have

\begin{theorem}
Let $\Lambda_{s}(t_{j})$ be a multiple eigenvalue and $f\in W,$ where
$j=v+1,v+2,...,u.$ Then%
\begin{equation}
\int\limits_{\lbrack-h,h]}\sum_{k\in\mathbb{N}(0)}a_{k}^{\pm}(t)\Psi
_{k,t}(x)dt=\sum_{k\in\mathbb{N}(0)}\int\limits_{r(0,h)}a_{k}^{\pm}%
(t)\Psi_{k,t}(x)dt+ \tag{82}%
\end{equation}%
\[
\sum_{j=v+1}^{u}\left(  F^{\pm}(s,j,x)+\sum_{k\in\mathbb{B}(s,j)}%
\int\limits_{(t_{j}-\varepsilon,t_{j}+\varepsilon)}a_{k}^{\pm}(t)\Psi
_{k,t}(x)dt\right)  ,
\]
where $r(0,h)=(-h,h)\backslash\left(  \bigcup_{j=v+1}^{u}(t_{j}-\varepsilon
,t_{j}+\varepsilon)\right)  $ and $F^{\pm}(s,j,x)$ is defined by (49).
\end{theorem}

In the same way we prove the following theorem by introducing the following
notation. Let $t_{u+1},t_{u+2},...,t_{p}$ be the SQ that lie in $(\pi
-h,\pi+h)$ and the operator $L_{t_{j}}$ has a multiple eigenvalue lying inside
$\Gamma(\pi)$.

\begin{theorem}
Let $\ \Lambda_{s}(t_{j})$ be a multiple eigenvalue and $f\in W,$ where
$j=u+1,u+2,...,p.$ Then%
\begin{equation}
\int\limits_{\lbrack\pi-h,\pi+h]}\sum_{k\in\mathbb{N}(\pi)}a_{k}^{\pm}%
(t)\Psi_{k,t}(x)dt=\sum_{k\in\mathbb{N}(\pi)}\int\limits_{r(\pi,h)}a_{k}^{\pm
}(t)\Psi_{k,t}(x)dt+ \tag{83}%
\end{equation}%
\[
\sum_{j=u+1}^{p}\left(  F^{\pm}(s,j,x)+\sum_{k\in\mathbb{B}(s,j)}%
\int\limits_{(t_{j}-\varepsilon,t_{j}+\varepsilon)}a_{k}^{\pm}(t)\Psi
_{k,t}(x)dt\right)  ,
\]
where $r(0,h)=(\pi-h,\pi+h)\backslash\left(  \bigcup_{j=u+1}^{p}%
(t_{j}-\varepsilon,t_{j}+\varepsilon)\right)  $ and $F^{\pm}(s,j,x)$ is
defined by (49).
\end{theorem}

Now using (77), (82) and Remark 1 and arguing as in the proof of Theorem 5 we obtain

\begin{theorem}
Let $f\in W$. \ Then the following equality holds
\begin{equation}
\int\limits_{(-h,h)}f_{t}(x)dt=\sum\limits_{k\geq N(0,h)}\int\limits_{[-h,h]}%
\left[  \sum_{s=\pm k}^{{}}\sum_{j=1}^{m}a_{s,j}(t)\Psi_{s,j,t}(x)\right]
dt+\sum_{k\in\mathbb{N}(0)}\int\limits_{r(0,h)}a_{k}(t)\Psi_{k,t}(x)dt+
\tag{84}%
\end{equation}%
\[
\sum_{j=v+1}^{u}\left(  \sum_{k\in\mathbb{B}(s,j)}\int\limits_{(t_{j}%
-\varepsilon,t_{j}+\varepsilon)}a_{k}(t)\Psi_{k,t}(x)dt+\int\limits_{(t_{j}%
-\varepsilon,t_{j}+\varepsilon)}\left[  \sum_{k\in\mathbb{S}(s,j)}a_{k}%
(t)\Psi_{k,t}\right]  dt\right)  .
\]
The series in (84) converge in the $L_{2}^{m}(a,b)$ norm for any
$a,b\in\mathbb{R}$ and (55) holds.
\end{theorem}

In the same way from (78) and (83) we obtain

\begin{theorem}
Let $f\in E.$ Then the following equality holds
\begin{equation}
\int\limits_{(\pi-h,\pi+h)}f_{t}(x)dt=\sum\limits_{k\geq N(\pi,h)}%
\int\limits_{[\pi-h,\pi+h]}\left[  \sum_{s=k,-k-1}^{{}}\sum_{j=1}^{m}%
a_{s,j}(t)\Psi_{s,j,t}(x)\right]  dt+\sum_{k\in\mathbb{N}(\pi)}\int
\limits_{r(\pi,h)}a_{k}(t)\Psi_{k,t}(x)dt+ \tag{85}%
\end{equation}%
\[
\sum_{j=u+1}^{p}\left(  \sum_{k\in\mathbb{B}(s,j)}\int\limits_{(t_{j}%
-\varepsilon,t_{j}+\varepsilon)}a_{k}(t)\Psi_{k,t}(x)dt+\int\limits_{(t_{j}%
-\varepsilon,t_{j}+\varepsilon)}\left[  \sum_{k\in\mathbb{S}(s,j)}a_{k}%
(t)\Psi_{k,t}\right]  dt\right)  .
\]
The series in (85) converge in the $L_{2}^{m}(a,b)$ norm for any
$a,b\in\mathbb{R}$ and (55) holds
\end{theorem}

Now using (41), Corollary 1 and Theorems 5, 11,12 we obtain

\begin{theorem}
For each $f\in W$ the spectral expansion defined by (41), (40), (54), (84) and
(85) holds.
\end{theorem}

\section{Conclusion}

In this section, we explain why we need to use square brackets in the spectral
expansion formulas (54), (84), and (85) for $L$. We also show that the number
of terms within these square brackets is minimal. First, consider the square
brackets in the formulas (54), (84) and (85) which contain the terms with
indices $k\in\mathbb{S}(s,j).$ By the definition of $\mathbb{S}(s,j)$ the
terms $a_{k}(t)\Psi_{k,t}$ \ for $k\in\mathbb{S}(s,j)$ are not integrable over
$(t_{j}-\varepsilon,t_{j}+\varepsilon),$ while their sum is integrable.
Moreover, in the general case, it is impossible to divide the set
$\mathbb{S}(s,j)$ into two subsets $\mathbb{S}^{\prime}(s,j)$ and
$\mathbb{S}(s,j)\backslash\mathbb{S}^{\prime}(s,j)$ so that the sums of the
terms $a_{k}(t)\Psi_{k,t}$ \ for $k\in$ $\mathbb{S}^{\prime}(s,j)$ and for
$k\in\mathbb{S}(s,j)\backslash\mathbb{S}^{\prime}(s,j)$ are integrable due to
the following. The terms $a_{k}(t)\Psi_{k,t}$ \ for all $k\in\mathbb{S}(s,j)$
correspond to the same ESS $\Lambda_{s}(t_{j})$ and the number of terms in
$\mathbb{S}(s,j)$ is less or equal to the multiplicity of the eigenvalue
$\Lambda_{s}(t_{j}).$ If the multiplicity of this eigenvalue is two or three,
then such a division is not impossible, since at least one of $\mathbb{S}%
^{\prime}(s,j)$ and $\mathbb{S}(s,j)\backslash\mathbb{S}^{\prime}(s,j)$
contains only one index $k$ for which $a_{k}(t)\Psi_{k,t}$ is not integrable
over $(t_{j}-\varepsilon,t_{j}+\varepsilon).$ Thus in the prevalent cases such
a division is not impossible. Maybe in the other cases such a division is
possible, however we can not determine the multiplicity of the eigenvalue
$\Lambda_{s}(t_{j}),$ since this is a small eigenvalue. Therefore, we compound
together the nonintegrable terms $a_{k}(t)\Psi_{k,t}$ \ for all $k\in
\mathbb{S}(s,j)$.

Now we consider the square brackets in (84) which compound the terms
$a_{s,j}(t)\Psi_{s,j,t}(x)$ with indices $(s,j)$ from the set $\mathbb{S}%
(k)=\left\{  \left(  \pm k,1\right)  ,(\pm k,2),...,(\pm k,m)\right\}  $ by
using the simplest case $L(C).$ Consideration the square brackets in (85)
which compound the terms with indices $(s,j)$ from the set $\left\{  \left(
k,1\right)  ,(k,2),...,(k,m),\left(  -k-1,1\right)  ,\left(  -k-1,2\right)
,...,\left(  -k-1,m\right)  \right\}  $ is the same. The eigenvalue
\[
\mu_{k,j}(t)=\left(  2\pi ki+ti\right)  ^{n}+\mu_{j}\left(  2\pi ki+ti\right)
^{n-2}%
\]
of $L_{t}(C)$ coincides with the eigenvalue $\mu_{-k,s}(t)$ if
\[
\left(  2\pi ki+ti\right)  ^{n}+\mu_{j}\left(  2\pi ki+ti\right)
^{n-2}=\left(  -2\pi ki+ti\right)  ^{n}+\mu_{s}\left(  -2\pi ki+ti\right)
^{n-2}.
\]
Solving this equality for the case $n=2$ we see that if $t=\frac{\mu_{j}%
-\mu_{s}}{4\pi(2k)}$, then $\mu_{k,j}(t)$ is a multiple eigenvalue. If $k$ is
a large number then $t$ belong to the neighborhood of $0.$ Similarly if
\[
t=\pi+\frac{\mu_{j}-\mu_{s}}{4\pi(2k+1)}%
\]
then the eigenvalue $\mu_{-k-1,s}(t)$ coincides with $\mu_{k,s}(t)$ and $t$
belong to the neighborhood of $\pi\ $if $k$ is a large number. This shows that
the eigenvalues $\lambda_{\pm k,1}(t),$ $\lambda_{\pm k,2}(t),...,\lambda_{\pm
k,m}(t)$ of the operator of the perturbed operator $L_{t}$ for some values of
$t\in(-h,h)$ may become ESS. Moreover, in the general case, the set
$\mathbb{S}(k)=\left\{  \left(  \pm k,1\right)  ,(\pm k,2),...,(\pm
k,m)\right\}  $ can not be divided into two disjoint subset $\mathbb{S}%
^{\prime}(k)$ and $\mathbb{S}(k)\backslash\mathbb{S}^{\prime}(k)$ such that
\[%
{\displaystyle\sum\limits_{(s,j)\in\mathbb{S}^{\prime}}}
a_{s,j}(t)\Psi_{s,j,t}(x)\text{ }\And\text{ }%
{\displaystyle\sum\limits_{s,j\in\mathbb{S}\backslash\mathbb{S}^{\prime}}}
a_{s,j}(t)\Psi_{s,j,t}(x)
\]
are integrable over $(-h,h)$, due to the following. For any proper subset
$\mathbb{S}^{\prime}(k)$ of $\mathbb{S}(k)$, it may exist $\left(  s,j\right)
\in\mathbb{S}^{\prime}(k)$ and $\left(  p,i\right)  \in\mathbb{S}%
(k)\backslash\mathbb{S}^{\prime}(k)$ such that the equality $\lambda
_{s,j}(t)=\lambda_{p,i}(t)$ holds for some value of $t_{0}\in(-h,h)$. In
addition $\lambda_{s,j}(t_{0})$ may become ESS and the expressions
$a_{s,j}(t)\Psi_{s,j,t}(x)$ and $a_{p,i}(t)\Psi_{p,i,t}(x)$ may become
nonintegrable on some neighborhood of $t_{0},$ while their sum is integrable.
Therefore we compound together the terms $a_{s,j}(t)\Psi_{s,j,t}(x)$ for all
$\left(  s,j\right)  \in\mathbb{S}$ in the spectral expansion (84) of
$\int\limits_{(-h,h)}f_{t}(x)dt.$

Another reason for using the square bracket (84) which compound together the
terms $a_{s,j}(t)\Psi_{s,j,t}(x)$ with indices $(s,j)$ from the set
$\mathbb{S}(k)$ is connected for the convergence of the series in (84). It is
possible that the integral
\[
\int\limits_{(-h,h)}a_{s,j_{s}}(t)\Psi_{s,j_{s},t}(x)dt\text{ }%
\]
exists, however its $L_{2}^{m}(a,b)$ norm does not approach zero as
$s\rightarrow\infty$. Therefore the series (84) does not converge in the norm
of $L_{2}^{m}(a,b)$ without square bracket. This is connected with the ESS at
infinity defined as follows.

\begin{definition}
We say that the operator $L$ has ESS at infinity if there exist sequence of
integers $\left(  k_{s},j_{s}\right)  $ and sequence of closed subsets $I(s)$
of $[0,2\pi)$ such that
\[
\lim_{s\rightarrow\infty}\int\nolimits_{I(s)}\left\vert \alpha_{k_{s},j_{s}%
}(t)\right\vert ^{-1}dt=\infty.
\]

\end{definition}

Hence, one can obtain a spectral expansion without parenthesis if and only if
$L$ has no ESS (or equivalently has no SQ) and ESS at infinity. This implies
that, in the general case, it is necessary to use the parenthesis, since the
operator $L$ has ESS and ESS at infinity. In fact the detailed analysis (see
[2] and [13, Section 2.7]) of the simplest case $m=1$ and $n=2$ shows that the
Hill's operator, in the general case, has spectral singularity, ESS and ESS at
infinity. Thus the nature of the spectral expansion problem for the operator
$L$ requires the use of parenthesis.

\end{document}